\documentclass[a4paper]{amsart}
\usepackage[paperheight=11in,paperwidth=8.5in]{geometry}
\usepackage{amsthm,amssymb,amsmath,hyperref, array, amsrefs}
\DefineSimpleKey{bib}{primaryclass}{}
\DefineSimpleKey{bib}{archiveprefix}{}

\BibSpec{arXiv}{%
  +{}{\PrintAuthors}{author}
  +{,}{ \textit}{title}
  +{}{ \parenthesize}{date}
  +{,}{ arXiv }{eprint}
  +{,}{ primary class }{primaryclass}
}
\usepackage{color, mdframed}
\usepackage{tikz}
\usetikzlibrary{decorations.pathreplacing}
\usepackage[utf8]{inputenc}
\usepackage{graphicx,enumerate}
\usepackage{mathtools,bm}
\usepackage{enumitem, booktabs}
\usepackage{amsfonts}

\allowdisplaybreaks
\makeatletter
\def\l@section{\@tocline{1}{12pt plus2pt}{0pt}{}{\bfseries}}
\def\l@subsection{\@tocline{2}{0pt}{2pc}{2pc}{}}
\makeatother
\makeatletter
\def\subsection{\@startsection{subsection}{2}{\z@}%
	{-3.25ex\@plus -1ex \@minus -.2ex}%
	{1.5ex \@plus .2ex}%
	{\normalfont\bfseries\boldmath}}
\def\subsubsection{\@startsection{subsubsection}{3}%
	\z@{.5\linespacing\@plus.7\linespacing}{-.5em}%
	{\normalfont\bfseries\boldmath}}
\renewcommand\paragraph{\@startsection{paragraph}{4}{\z@}%
	{3.25ex \@plus1ex \@minus.2ex}%
	{-1em}%
	{\normalfont\normalsize\bfseries}}
\makeatother

\theoremstyle{plain}
\newtheorem{thm}{Theorem}[section]

\newtheorem{cor}[thm]{Corollary}
\newtheorem{lem}[thm]{Lemma}
\newtheorem{prop}[thm]{Proposition}

\theoremstyle{definition}
\newtheorem{defn}[thm]{Definition}

\theoremstyle{remark}
\newtheorem{rem}[thm]{Remark}

\theoremstyle{plain}

\numberwithin{equation}{section}

\theoremstyle{plain} 
\newcommand{\thistheoremname}{}
\newtheorem{genericthm}[thm]{\thistheoremname}

  \newtheorem*{genericthm*}{\thistheoremname}
\newenvironment{namedthm*}[1]
  {\renewcommand{\thistheoremname}{#1}%
   \begin{genericthm*}}
  {\end{genericthm*}}

\newcommand{\B}{{\mathbb B}}
\newcommand{\D}{{\mathbb D}}

\newcommand{\R}{{\mathbb R}}

\newcommand{\C}{{\mathbb C}}

\newcommand{\N}{{\mathbb N}}

\newcommand{\im}{{\rm Im}\,}
\newcommand{\re}{{\rm Re}\,}

\newcommand{\calA}{{\mathcal A}}
\newcommand{\calB}{{\mathcal B}}

\newcommand{\calD}{{\mathcal D}}

\makeatletter
\newcommand{\vast}{\bBigg@{4}}
\newcommand{\Vast}{\bBigg@{5}}
\makeatother

\newcommand{\frakA}{{\mathfrak A}}

\def\udot#1{\ifmmode\oalign{$#1$\crcr\hidewidth.\hidewidth
    }\else\oalign{#1\crcr\hidewidth.\hidewidth}\fi}

\def\R{\mathbb{R}}

\def\C{\mathbb{C}}

\def\beq{\begin{equation}}
\def\eeq{\end{equation}}
\makeatletter
\newcommand{\doublewidetilde}[1]{{%
  \mathpalette\double@widetilde{#1}%
}}
\newcommand{\double@widetilde}[2]{%
  \sbox\z@{$\m@th#1\widetilde{#2}$}%
  \ht\z@=.9\ht\z@
  \widetilde{\box\z@}%
}
\makeatother

\makeatletter
\def\@makefnmark{%
  \leavevmode
  \raise.9ex\hbox{\fontsize\sf@size\z@\normalfont\tiny\@thefnmark}}
\makeatother

\begin{document}
	
\title[]{Boundedness and compactness of Bergman projection commutators in two-weight setting}
\author[]{Bingyang Hu, Ji Li, and Nathan A. Wagner}

\address{Bingyang Hu: Department of Mathematics and Statistics, Auburn University, 221 Parker Hall, Auburn, 221 Parker Hall, Auburn, AL 36849, USA, \thanks{Orcid ID: 0000-0001-8005-5666}}%
\email{bzh0108@auburn.edu}

\address{Ji Li: School of Mathematical and Physical Sciences, Macquarie University, NSW 2109, Australia, \thanks{Orcid ID: 0000-0003-0995-3054}}%
\email{ji.li@mq.edu.au}

\address{Nathan A. Wagner: Department of Mathematics, Brown University, 151 Thayer Street, Providence, RI 02906, USA, \thanks{Orcid ID: 0000-0003-0096-1541}}%
\email{nwagner8@gmu.edu}

\begin{abstract}
The goal of this paper is to study the boundedness and compactness of the Bergman projection commutators in two weighted settings via the weighted BMO and VMO spaces, respectively. The novelty of our work lies in the distinct treatment of the symbol $b$ in the commutator, depending on whether it is analytic or not, which turns out to be quite different. In particular, we show that an additional weight condition due to Aleman, Pott, and Reguera is necessary to study the commutators when $b$ is not analytic, while it can be relaxed when $b$ is analytic. In the analytic setting, we completely characterize boundedness and compactness, while in the non-analytic setting, we provide a sufficient condition which generalizes the Euclidean case and is also necessary in many cases of interest. Our work initiates a study of the commutators acting on complex function spaces with different symbols. \\

\end{abstract}
\date{\today}

\thanks{\textbf{Funding Acknowledgement:} The first author was supported by the Simons Travel grant MPS-TSM-00007213. The second author is supported by Australian Research Council DP 220100285. The third author was supported by a National Science Foundation MSPRF DMS-2203272. }

\maketitle


\section{Introduction}
 
To illustrate the main idea, we first consider the most convenient case. Let $\R_{+}^2:=\left\{z: \im z>0\right\}$ be the upper half plane. For any $\alpha>-1$, the \emph{weighted $L_{\alpha}^2(\R_+)$ space} is given by 
$$
L_\alpha^2(\R_+^2):=\left\{f \ \textrm{measurable}: \left\|f \right\|^2_{L^2_\alpha(\R_+^2)}:=\int_{\R_{+}^2} |f(z)|^2 dA_\alpha(z)<+\infty \right\},
$$
where $dA_\alpha(z):=\frac{1}{\pi} (\alpha+1) \left(2 \im z \right)^{\alpha} dA(z)$ and $dA(z)=dxdy$ is the standard Lebesgue measure on $\R_{+}^2$. Here and henceforth, we will mostly suppress notation on the underlying measure parameter $\alpha$, though some of the function spaces we will define may depend on $\alpha.$ Let $H(\R_{+}^2)$ be the space of all holomorphic functions on $\R_{+}^2$ and
$$
A_\alpha^2(\R_{+}^2):=L_\alpha^2(\R_+^2) \cap H(\R_{+}^2). 
$$
The \emph{weighted Bergman projection} $L_\alpha^2(\R_+^2) \mapsto A_\alpha^2(\R_+^2)$ is given by 
$$
P_\alpha f(z):=i^{2+\alpha} \int_{\R_{+}^2} \frac{f(w)}{(z-\overline{w})^{2+\alpha}} dA_\alpha(w),
$$
and the \emph{weighted Berezin transform} is given by 
$$
P_\alpha^{+} f(z):=\int_{\R_{+}^2} \frac{f(w)}{\left| z-\overline{w} \right|^{2+\alpha}} dA_\alpha(w). 
$$
It is well known that $P_\alpha^{+}: L_\alpha^2(\R_+^2) \mapsto L_\alpha^2(\R_+^2)$. 

\vspace{0.1cm}

The aim of this paper is to study the behavior of the commutators $\left[b, P_\alpha\right]$ and $\left[b, P^+_\alpha\right]$, where $b$ is a measurable function on $\mathbb{R}_{+}^2$, within a weighted framework. To the best of our knowledge, this is the first attempt in the literature to explore (two) weight estimates for commutators of Bergman projections in complex settings. This work is motivated by the recent development in harmonic analysis of the two-weight theory in the Calder\'on--Zygmund setting, which dates back to an early paper of Bloom \cite{Bloom1985}, where he studied the (two) weighted estimates of the commutator $[b, H]$. This line of research has been further investigated recently by many authors. Here, we briefly recall some closely related results:
\vspace{-0.05cm}
\begin{enumerate}
    \item [$\bullet$] two-weight \emph{boundedness} for the commutator of \emph{Riesz transforms}, due to Holmes, Lacey and Wick \cite{HLW2017} using a new method via representation formula from Hyt\"onen  and decomposition via paraproducts;
    \item [$\bullet$] two-weight \emph{boundedness} for the commutator of \emph{Calder\'on--Zygmund operators}, due to Lerner, Ombrosi, and Rivera-R\'ios \cite{Lerner2017} via sparse domination, and due to Hyt\"onen \cite{Hytonen2021} by weak factorization;
    \item [$\bullet$] two-weight \emph{compactness} for \emph{Riesz transforms}, due to Lacey and the second author \cite{LaceyLi2022} via real median method;
\end{enumerate}
Another key heuristic comes from classical complex function theory, where much attention has been given to the unweighted theory of boundedness and compactness of the Bergman commutators, or equivalently, Hankel operators (for most of the cases). Here are some closely related ones:
\vspace{-0.02cm}
\begin{enumerate}
    \item [$\bullet$] boundedness and compactness of the Hankel operator \emph{$H_{\overline{b}}$ on the unit disc $\D$},  due to Axler \cite{Axler1986};

    \item [$\bullet$] boundedness and compactness of the Hankel operators \emph{$H_b$ and $H_{\overline{b}}$ on the unit ball $\B$}, due to Zhu \cite{Zhu1992};

    \item [$\bullet$] generalization of the above results in \emph{bounded symmetric domains}, due to B\'ekoll\'e, Berger, Coburn, and Zhu \cite{BBCZ1990};

    \item [$\bullet$] generalization of the above results in \emph{strongly pseudoconvex domains with smooth boundary}, due to Li \cite{Li1992, Li1994};

    \item [$\bullet$] generalization of the above results in strongly pseudoconvex \emph{domains with minimal smoothness boundary regularity}, due to Huo, Lanzani, Palencia, first and the third authors \cite{HHLPW2024}; 

    \item [$\bullet$] generalization of the above results on domains with bounded intrinsic geometry, due to Zimmer \cite{Zimmer2023}.
\end{enumerate}
Our ultimate goal is to combine the above two themes and understand how the techniques from the modern dyadic harmonic analysis can be adapted to the complex setting. 

\vspace{0.05cm}

Returning to our current setup, we begin by introducing several definitions. Let $\calD^1$ and $\calD^2$ be any adjacent dyadic systems on $\R$ (namely, the two dyadic systems on $\R$ which satisfies the $1/3$-trick). Denote $\calD:=\calD^1 \cup \calD^2$. For each $I \in \calD$ (or more generally any interval $I \subset \R$), we let 
\begin{equation} \label{20250121eq21}
Q_I:=\left\{z=x+iy \in \R_{+}^2: x \in I, y \in (0, |I|) \right\}
\end{equation} 
be the \emph{Carleson box} associated to $I$, and  
\begin{equation} \label{20250121eq22}
Q_I^{up}:=\left\{z=x+iy \in \R_{+}^2: x \in I, y \in \left(\frac{|I|}{2}, |I| \right) \right\}.
\end{equation}
be the \emph{upper Carleson box} associated with $I$. We write $c_I$ to denote the Euclidean center of the square $Q_I$. Now we introduce the \emph{dyadic B\'ekolle--Bonami $\calB_2$ weights}, which can be viewed as a natural analog of the Muckenhoupt $A_2$ weights in the real setting. 

\begin{defn} [Dyadic B\'ekolle--Bonami $\calB_2$ weights on $\R_+^2$]
Let $\mu$ be a weight on $\R_{+}^2$, we say $\mu$ belongs to the \emph{dyadic B\'ekolle--Bonami $\calB_2$ weight class} $\R_+^2$ if 
    \begin{equation} \label{20250106eq01}
    \left[\mu \right]_{\calB_2}:=\sup_{I \in \calD} \left( \frac{1}{A_\alpha(Q_I)} \int_{Q_I} \mu(z) dA_\alpha(z) \right) \left( \frac{1}{A_\alpha(Q_I)} \int_{Q_I} \mu^{-1}(z) dA_\alpha(z) \right)<+\infty. 
    \end{equation} 
\end{defn}

\begin{rem}
As a direct consequence of the definition of $\calD$, it is clear that $I \in \calD$ can be replaced by $I \subseteq \R$ in \eqref{20250106eq01}. Moreover, we note that the collection $\{Q_I\}_{I \in \calD}$ is also a \emph{Muckenhoupt basis}, that is, a collection of dyadic cubes such that the associated maximal operator enjoys the weighted $L^p$ estimates (see, e.g., \cite[Proposition 3.13]{GHK2022}). 
\end{rem}

Next, we define the associated weighted BMO spaces adapted to the Carleson tents $\{Q_I\}_{I \in \calD}$. Let $\alpha>-1$ be fixed. Given a weight function $\nu$ and a measurable subset $E \subset \R_{+}^2$, we let $\nu(E):= \int_E \nu \, dA_\alpha$, and $b_{Q_I}$ be the average of $b$ on $Q_I$ with respect to the measure $dA_\alpha.$  The \emph{weighted BMO space ${\rm BMO}_{\nu}(\R_{+}^2)$}  is given by 
\begin{defn}\label{def-BMOA}
Suppose $\nu\in \calB_2$. For $b \in L_{ loc}^1(\R_{+}^2, dA_\alpha)$, we say $b \in {\rm BMO}_{\nu} \left(\R_{+}^2 \right)$ if 
$$
\left\|b \right\|_{{\rm BMO}_{\nu}(\R_{+}^2)}:=\sup_{I \in \calD} {1\over \nu(Q_I)}\int _{Q_I} |b(z)-b_{Q_I}|dA_\alpha(z)<+\infty.
$$
Moreover, if $b \in H(\R_{+}^2)$, we say $b \in {\rm BMOA}_{\nu}\left(\R_{+}^2 \right).$
\end{defn}

Our first main theorem states as follows, which studies the boundedness of the commutators: 
\begin{thm} \label{mainthm001}
Let $\mu, \lambda \in \calB_2$ and set $\nu=\mu^{\frac{1}{2}}\lambda^{-\frac{1}{2}}$. Let further, $b \in L_{ loc}^1\left(\R_{+}^2, dA_\alpha \right) \cap H(\R_{+}^2) $. Then the commutators $[b, P_\alpha]$ and $[b, P^+_\alpha]$ are  bounded from  $L_\alpha^2 \left(\R_{+}^2, \mu \right)$ to $L_\alpha^2 \left(\R_{+}^2, \lambda \right) $ if and only if $b \in {\rm BMOA}_{\nu} \left(\R_{+}^2 \right)$. Moreover, we have the following Bloom-type estimates: 
$$ 
\|[b, P_\alpha]\|_{L_\alpha^2 \left(\R_{+}^2, \mu \right) \mapsto L_\alpha^2 \left(\R_{+}^2, \lambda \right) }, \; \|[b, P^+_\alpha]\|_{L_\alpha^2 \left(\R_{+}^2, \mu \right) \mapsto L_\alpha^2 \left(\R_{+}^2, \lambda \right) } \simeq \|b\|_{ {\rm BMOA}_{\nu} \left(\R_{+}^2 \right) }.
$$ 
\end{thm}
 
For compactness, we define the \emph{weighted VMO space ${\rm VMO}_{\nu}(\R_{+}^2)$} as follows:
\begin{defn}\label{def-VMOA}
Suppose $\nu\in \calB_2$. For $b \in {\rm BMO}_{\nu}\left(\R_{+}^2 \right)$, we say $b \in {\rm VMO}_{\nu}(\R_{+}^2)$ if 
\begin{align*}
&\lim_{|I| \rightarrow 0}\sup_{Q_I} {1\over \nu(Q_I)}\int _{Q_I} |b(z)-b_{Q_I}|dA_\alpha(z)= 0,\\
&\lim_{|c_I| \rightarrow \infty}\sup_{Q_I} {1\over \nu(Q_I)}\int _{Q_I} |b(z)-b_{Q_I}|dA_\alpha(z)= 0.
\end{align*}
If, in addition to satisfying the above, $b$ is holomorphic, then we say $b \in {\rm VMOA}_{\nu}\left(\R_{+}^2 \right).$
\end{defn}

Our second main theorem deals with the compactness of the commutators: 
\begin{thm}\label{mainthm002}
Let $\mu, \lambda \in \calB_2$  and set $\nu=\mu^{\frac{1}{2}}\lambda^{-\frac{1}{2}}$. Let further, $b \in {\rm BMOA}_{\nu} \left(\R_{+}^2 \right)$. Then the commutators $[b, P_\alpha]$ and $[b, P^+_\alpha]$ are compact from  $L_\alpha^2 \left(\R_{+}^2, \mu \right)$ to $L_\alpha^2 \left(\R_{+}^2, \lambda \right) $ if and only if $b \in {\rm VMOA}_{\nu}\left(\R_{+}^2 \right)$. 
\end{thm}

\begin{rem}
For the lower bound in Theorem \ref{mainthm001}, the proof shows the holomorphic assumption is not needed and we can simply assume $b \in \textnormal{BMO}_\nu \left(\R_+^2 \right).$ A similar statement holds true for Theorem \ref{mainthm002}.
\end{rem}

In the second part of this paper, our goal is to relax the analytic assumption on $b$, namely, we merely assume \textit{$b$ is a measurable function belonging to certain BMO classes}. It turns out that this is a more subtle situation. The main reason is that the system of dyadic Carleson boxes $\{Q_I\}_{I \in \calD^i}$ is \emph{not} a full dyadic resolution of $\R_+^2$, contrast to the standard dyadic systems in the real setting; in particular, 
\begin{enumerate} 
\item [$\bullet$] it does \emph{not} contain small dyadic cubes that are far from the boundary $\{z: \im z=0\}$;
\item [$\bullet$] and any Calder\'on--Zygmund type decomposition associated to $\{Q_I\}_{I \in \calD^i}$ \emph{fails} in this situation.
\end{enumerate} 
Consequently, $\calB_2$ weights (or more generally, $\calB_p$ weights\footnote{The weight class $\calB_p$ can be defined analogously to the Muckenhoupt $A_p$ weights. However, since this paper focuses on the weighted $L^2$ theory, we omit the detailed definition of $\calB_p$ weights.}) can be ill-behaved away from the boundary and do \emph{not} enjoy the classical properties of $A_\infty$ weights in the Euclidean setting such as a reverse H\"{o}lder inequality. 

\begin{rem}
In the case when $b$ is a holomorphic symbol (namely, the situations that were considered in Theorems \ref{mainthm001} and \ref{mainthm002}), the $\calB_2$ weight condition indeed suffices. This is because the behavior of a holomorphic function is uniquely determined by its boundary behavior along $\R$, which is precisely the subject of the Definitions \ref{def-BMOA} and \ref{def-VMOA}. 
\end{rem}

To obtain positive results for the commutator $[b, P_\alpha]$ for $b$ being measurable, we need to impose an extra condition on the weights $\mu, \lambda$ which was first introduced by Aleman, Pott, and Reguera (see, \cite{APR2019}).

\begin{defn} [Aleman--Pott--Regeura weights ($\mathcal{APR}$)]
Let $\mu$ be a weight on $\R_{+}^2$, we say $\mu$ is a \emph{Aleman--Pott--Regeura weight ($\mathcal{APR}$)} or is of \emph{bounded hyperbolic oscillation} if there exists some $C_\mu>0$, such that for any $I \subset \R$ and $z, w \in Q_I^{up}$, there holds
$$
C_{\mu}^{-1} \mu(w) \le \mu(z) \le C_\mu \mu(w). 
$$
\end{defn}

Recall that if $\mu \in \calB_2 \cap \mathcal{APR}$, then $\mu$ enjoys the \emph{reverse H\"older estimate}: there exists some $r_\mu>1$, such that for any $Q_I$ with $I \in \calD$, one has 
\begin{equation} \label{20240718eq69}
\left(\frac{1}{A_\alpha(Q_I)} \int_{Q_I} \left(\mu(z) \right)^{r_\mu} dA_\alpha(z) \right)^{1/r_\mu} \lesssim \frac{1}{A_\alpha(Q_I)} \int_{Q_I} \mu(z) dA_\alpha(z).
\end{equation}

In this situation, we obtain a sufficient condition for a one-weight inequality (the two-weight inequality does not naturally generalize). However, we \emph{do} need to assume the $\mathcal{APR}$ regularity condition on the weight $\sigma$ in this case, and this assumption is \emph{sharp} in the sense that the conclusion does not hold for generic $\calB_2$ weights.  More precisely, we refer the reader to Section \ref{20250107subsec01} for an example of a weight $\sigma \in \calB_2 \backslash \mathcal{APR}$ such that $[b, P_\alpha]$ fails to be bounded on $L_\alpha^2\left(\R_+^2, \sigma \right)$. The scheme of proof in this case is different compared to Theorems \ref{mainthm001} and \ref{mainthm002}, and parallels classical complex function theory. 

\vspace{0.1cm}
 
To state our main result, we introduce an $L^2$ variant of the BMO norm in the unweighted case. Unlike the classical Euclidean setting, we observe that the natural BMO space adapted to Bergman spaces does \emph{not} satisfy the John--Nirenberg inequality, because they can have singularities away from the boundary that prevent membership in $L^p_{loc}(\R_{+}^2)$ for large $p$. However, in the special case of note when the symbol $b$ is an analytic function, one can check the BMO oscillation condition is independent of $p$ and is equivalent to saying $b$ belongs to an analog of the Bloch space. See \cite{Si2022} for the definition of Bloch space and the relationship of the Bloch space to BMO in this setting and a form of the John--Nirenberg inequality (actually the more general setting of the Siegel upper-half space). 

\begin{defn}\label{def-BMO}
For $b \in L^2(\R_{+}^2)$, we say $b \in {\rm BMO}^2 \left(\R_{+}^2 \right)$ if 
$$
\left\|b \right\|_{{\rm BMO}^{2}(\R_{+}^2)}:=\sup_{I \in \calD} \left({1\over A(Q_I)}\int _{Q_I} |b(z)-b_{Q_I}|^2 dA(z) \right)^{1/2}<+\infty.
$$
\end{defn}

Here are some remarks for the above definition. 

\begin{rem} \label{EquivalentBMO}
\begin{enumerate}
\item In the definition above, we could replace the averages of $b$ over $Q_I$ with averages over disks of fixed radius in the hyperbolic metric. It is easy to show the BMO condition on Carleson tents implies the condition on hyperbolic balls, while the other direction follows from a geometric summation argument and the fact that the BMO space associated with hyperbolic balls is independent of the hyperbolic radius $r$. We refer the reader \cite[Section 5.1]{HHLPW2024} for more detailed information about different notations of BMOs over Carleson boxes. Moreover, one can check that this BMO norm is quantitatively equivalent to the norm defined using hyperbolic disks. See \cite{Si2022} for an alternative characterization of this BMO space using the Berezin transform.
\item The integrals and averages in Definition \ref{def-BMO} are taken with respect to the planar Lebesgue area measure corresponding to $\alpha=0$ (this BMO space does not ``see'' the parameter $\alpha$).
\end{enumerate}
\end{rem} 
 
 The version of Theorem \ref{mainthm001} in the one-weight case with non-analytic symbol is as follows:

\begin{thm} \label{mainthm001dropanalytic}
Let $\sigma \in \calB_2 \cap \mathcal{APR}$. Then if $b \in {\rm BMO}^2(\R_{+}^2)$, the commutator $[b, P_\alpha]$ is bounded from  $L_\alpha^2 \left(\R_{+}^2, \sigma \right)$ to $L_\alpha^2 \left(\R_{+}^2, \sigma \right) $.  Moreover, we have
$$ \|[b, P_\alpha]\|_{L_\alpha^2 \left(\R_{+}^2, \sigma \right) \to L_\alpha^2 \left(\R_{+}^2, \sigma \right) } \lesssim \|b\|_{ {\rm BMO}^{2} \left(\R_{+}^2 \right) }.
$$
Finally, the $\mathcal{APR}$ condition is sharp in this case. 
\end{thm}
Observe that in this situation, the natural BMO space is unweighted and in fact coincides with the classical case of the unit disk \cite{Zhu07}, or the version discussed for the upper-half space in \cite{Si2022}. 

\begin{rem}
Theorem \ref{mainthm001dropanalytic} provides a sufficient condition on an $L^2$ symbol $b$ for the commutator to enjoy weighted bounds in the one weight setting. By the proof of Theorem \ref{mainthm001}, a necessary condition is given by 

$$
\sup_{I \in \calD} {1\over A(Q_I)}\int _{Q_I} |b(z)-b_{Q_I}| dA(z) <+\infty,$$
which for non-analytic symbols $b$ is strictly weaker than the $\rm{BMO}^2(\R_{+}^2)$ condition. This gap between necessary and sufficient conditions can be closed in certain important cases of interest; for example, the Lebesgue measure case, or more generally when the $\mathcal{B}_2$ weight is given  by an appropriate power of the Jacobian of a conformal map. See Theorem \ref{PartialConverse}  for more details. We do not know if this gap can be closed in the most general case when $\sigma \in \mathcal{B}_2.$ The lack of a John-Nirenberg inequality for BMO spaces in the Bergman setting appears to be a significant obstacle to obtaining the lower bound in the general setting, as the powerful median method yields a BMO condition with $L^1$ averages, not $L^2$ averages.

\end{rem}

\vspace{0.1cm}

The novelty of this work lies in the \emph{distinct} treatments when the symbol $b$ is either assumed to be analytic or merely measurable. In the first case, by using the mean value property of the analytic function, we prove our main results via a sparse domination approach inspired by recent developments in harmonic analysis. Moreover, we make use of a complex median method (introduced in \cite{Wei2024}), replacing the real median method used in the real case for Calder\'on--Zygmund type operators (see, e.g., \cite{LaceyLi2022}), to estimate the lower bound of the operator norm of the commutators. In the second case when $b$ is only assumed to be measurable, we investigate the problem differently via the behavior of the Hankel operators and prove that the $\mathcal{APR}$ condition is necessary in this case. To the best of our knowledge, our work furthers the line of the research on commutators associated with holomorphic projections acting on \emph{weighted} Lebesgue spaces with Muckenhoupt-type weights and with analytic or non-analytic symbols.

\vspace{0.1cm}

The structure of the current paper is as follows. Section \ref{s:2} and Section \ref{s:3} are devoted to proving the boundedness and compactness results Theorems \ref{mainthm001} and \ref{mainthm002} for $[b, P_\alpha]$ and $[b, P_\alpha^+]$, respectively, where $b$ is an analytic symbol. In Section \ref{s:4}, we study the boundedness of the commutators $[b, P_\alpha]$ without the analytic assumption on $b$. Moreover, we give an example to show that the $\mathcal{APR}$ condition is necessary in this case. Finally, in Section \ref{s:5}, we extend our main results to the unit ball $\B$ in $\C^n$. 

Throughout this paper, for $a ,b \in  \mathbb{R}$, $a\lesssim b$ means there exists a positive number $C$, which is independent of $a$ and $b$, such that $a\leq C\,b$. Moreover, if both $a \lesssim  b$ and $b\lesssim a$ hold, we say $a \simeq b$.

\medskip

\noindent {\bf Acknowledgment.} The authors thank Brett Wick and Ana \v{C}olovi\'c for their suggestions and helpful discussions on the sharpness of the $\mathcal{APR}$ condition in the non-analytic symbol case.

\bigskip
\section{Proof of Theorem \ref{mainthm001}}\label{s:2}

Our focus in this section is to prove the Bloom-type bounds for $[b, P_\alpha]$, while the proof of $[b, P_\alpha^+]$ follows from a simple modification, which we will comment at the end of both Sections \ref{20250116subsec01} and \ref{20250116subsec02}. Throughout the argument, the parameter $\alpha>-1$ will be fixed; hence, dependence on $\alpha$ will often be suppressed from the notation. 

To start with, we introduce the \emph{dyadic maximal operator} defined with respect to the dyadic grid $\calD^j$:
$$ M_{\calD^j} f(z):= \sup_{I \in \calD^j} \left(\frac{1}{A_\alpha(Q_I)}\int_{Q_I} |f| \, dA_\alpha \right) \, \mathbf{1}_{Q_I}(z), \quad f \in L^1_\alpha\left(\R_+^2 \right). $$
Similarly, we write $M_{\calD}$ to denote the dyadic maximal operator where the supremum is taken over $\calD^1 \cup \calD^2.$

Recall that the $\calB_2$ classes are all defined with respect to the underlying measure $A_\alpha.$ To prove the Bloom-type boundedness result for $[b, P_\alpha]$, we will modify an approach using sparse domination originally due to Lerner, Ombrosi, and Rivera-Rios in the Calder\'{o}n--Zygmund setting, see \cite{Lerner2017} . 

\subsection{Some sparse domination estimates}
We begin by collecting some sparse domination estimates, which are useful for proving the Theorem \ref{mainthm001}. First, given a fixed dyadic grid $\calD^j$, $j \in \{0,1\}$, we define the dyadic averaging operator
\begin{equation}
\mathcal{A}_{\calD^j}f(z):= \sum_{I \in \calD^j} \left( \frac{1}{A_\alpha(Q_I)} \int_{Q_I} |f| \, dA_\alpha \right) \mathbf{1}_{Q_I}(z),\quad f \in L^2_{\alpha}(\R_{+}^2)
\label{SparseOp}
\end{equation}
Note that since $A_\alpha(Q_I) \simeq_{\alpha} A_\alpha(Q_I^{up}) \simeq_\alpha |I|^{2+\alpha}$ for all $I \in \calD^j$ and the collection $\{Q_I^{up}\}_{I \in \calD^j}$ is pairwise disjoint, the operator $\mathcal{A}_{\calD^j}$ can be viewed as a ``pointwise sparse operator''. The following result establishes a pointwise sparse bound for the positive Bergman operator. 

\begin{prop}{\cite[Proposition 3.4]{PottReguera2013}} \label{BergSparse}
For any $f \in L_\alpha^2 \left(\R_+^2 \right)$, one has 
\begin{equation}
|P_\alpha^{+} f(z)| \lesssim \sum_{j=1,2} \mathcal{A}_{\calD^j}f(z),  \quad z \in \R_{+}^2.  \label{SparseDom}
\end{equation}
\end{prop}
As an immediate corollary, we obtain the following sparse domination for the commutator $[b, P_\alpha]$.

\begin{cor} \label{SparseCom}
The following pointwise bound holds for any compactly supported, bounded $f$ and $b \in L^1_{ loc}(\R_{+}^2, dA_\alpha)$:
\begin{align} 
|[b,P_\alpha]f(z)|&  \lesssim  \sum_{j=1,2}\sum_{I \in \calD^j} \left( \frac{1}{A_\alpha(Q_I)} \int_{Q_I} |b-b_{Q_I}| |f| \, dA_\alpha \right) \mathbf{1}_{Q_I}(z)\\
&\qquad +  \sum_{j=1,2}\sum_{I \in \calD^j} |b(z)-b_{Q_I}| \left( \frac{1}{A_\alpha(Q_I)}  \int_{Q_I} |f| \, dA_\alpha \right) \mathbf{1}_{Q_I}(z)\nonumber\\
&  =: \mathcal{A}_bf(z)+ \mathcal{A}_b^* f(z).\nonumber
\end{align}

\begin{proof}
We estimate, using the triangle inequality and Proposition \ref{BergSparse}:
\begin{align*}
|[b,P_\alpha]f(z)| & \leq \int_{\R_{+}^2} \frac{|f(w)| \, |b(z)-b(w)|}{|z-\overline{w}|^{2+\alpha}} dA_\alpha(w)= P_\alpha^{+}\left(  \left| (b(z)-b(\cdot))f(\cdot) \right| \right)(z)\\
& \lesssim \sum_{j=1,2} \sum_{I \in \calD^j} \left( \frac{1}{A_\alpha(Q_I)} \int_{Q_I} |f(w)| |b(z)-b(w)| \, dA_\alpha(w) \right) \mathbf{1}_{Q_I}(z)\\
& \leq \sum_{j=1,2} \sum_{I \in \calD^j} \left( \frac{1}{A_\alpha(Q_I)} \int_{Q_I} |f(w)| |b(w)-b_{Q_I}| \, dA_\alpha(w) \right) \mathbf{1}_{Q_I}(z)\\
& \qquad + \sum_{j=1,2} \sum_{I \in \calD^j} |b(z)-b_{Q_I}| \left( \frac{1}{A_\alpha(Q_I)} \int_{Q_I} |f(w)| \, dA_\alpha(w) \right) \mathbf{1}_{Q_I}(z)\\
& = \mathcal{A}_bf(z)+ \mathcal{A}_b^* f(z).
\end{align*}
The proof is complete.
\end{proof}

\end{cor}

The additional ingredient we need to prove Theorem \ref{mainthm001} is an easy estimate of the \emph{local oscillation} of a holomorphic function $b$ regarding the dyadic Carleson structure. The holomorphic property of $b$ is \emph{crucial} for our analysis here. Here and henceforth, let $\widetilde{Q}_J$ denote some fixed small inflation of the Carleson cube $Q_J$ obtained by expanding the interval $J$ to a slightly larger (non-dyadic) interval with the same center of $J$; this inflation is a technicality needed to exploit the mean-value property. 

\begin{lem} \label{OscLemma}
Let $b \in H(\R_{+}^2) \cap L^{1}_{loc}(\R_{+}^2, dA_\alpha)$, and $I \in \calD^j.$ Then for any $z \in Q_I$, one has\footnote{Our argument is also valid for any $b$ satisfying the sub-mean value property.}
\begin{equation}
|b(z)-b_{Q_I}| \lesssim \sum_{J \in \calD^j, J \subseteq I} \left( \frac{1}{A_\alpha(\widetilde{Q}_J)} \int_{\widetilde{Q}_J} |b(w)-b_{Q_J}| \, d A_\alpha(w) \right) \mathbf{1}_{Q_J}(z).
\end{equation}
\end{lem}

\begin{proof}
Fix $z \in Q_I.$  Let $k(z)$ be the unique non-negative integer and $I_{k(z)} \in \calD^j$ be the unique dyadic interval such that $|I_{k(z)}|=2^{-k(z)}|I|$, $I_{k(z)} \subseteq I$, and $z \in Q_{I_{k(z)}}^{up}$. Moreover, for any $\ell \leq k(z)$, let $I_\ell \in \calD^j$ denote the unique interval so that $z \in Q_{I_{\ell}}$, $|I_{\ell}|=2^{-\ell} |I|$, and $I_{k(z)} \subseteq I_{\ell} \subseteq I$. Note that $\{I_j\}_{j=0}^{k(z)}$ corresponds a chain of upper half tents that connected $z$ with $Q_I^{up}$. 

\vspace{0.05cm}

By the mean value property for holomorphic functions, it is easy to see
\begin{align*}
 |b(z)-b_{Q_{I_{k(z)}}}| & \lesssim \frac{1}{A_\alpha(\widetilde{Q}_{I_{k(z)}}^{up})} \int_{\widetilde{Q}_{I_{k(z)}}^{up}} |b-b_{Q_{I_{k(z)}}}| \, dA_\alpha \\
 & \lesssim \frac{1}{A_\alpha(\widetilde{Q}_{I_{k(z)}})} \int_{\widetilde{Q}_{I_{k(z)}}} |b-b_{Q_{I_{k(z)}}}| \, dA_\alpha.
   \end{align*}
Then using the above estimates with the triangle inequality and the doubling property of the $A_\alpha$ measure, we can estimate
\begin{align*}
|b(z)-b_{Q_I}|  & \leq |b(z)-b_{Q_{I_{k(z)}}}|+ \sum_{j=1}^{k(z)} |b_{Q_{I_j}}-b_{Q_{I_{j-1}}}| \\
&\lesssim \sum_{J \in \calD^j, J \subseteq I} \frac{1}{A_\alpha(\widetilde{Q}_J)} \int_{\widetilde{Q}_J} |b-b_{Q_J}| \, d A_\alpha.
\end{align*}
The proof is complete.
\end{proof}
We are now ready to prove Theorem \ref{mainthm001}, and we split the proof into two parts. 

\subsection{Sufficiency} \label{20250116subsec01}
Let $b \in {\rm BMOA}_{\nu}\left(\R_{+}^2 \right).$ Note that the operator $\mathcal{A}_b^*$ is the adjoint of $\calA_b$, under the usual $L^2_\alpha$ pairing, of the operator $\mathcal{A}_b$. Therefore, to prove that $[b,P_\alpha]$ is bounded from $L^2_\alpha(\R_{+}^2, \mu)$ to $L^2_\alpha(\R_{+}^2, \lambda)$ , by Corollary \ref{SparseCom} and duality, it suffices to prove that the operator $\mathcal{A}_b$ is bounded from $L^2_\alpha(\R_{+}^2, \mu)$ to $L^2_\alpha(\R_{+}^2, \lambda)$ as well as from $L^2_\alpha\left(\R_{+}^2, \lambda^{-1}\right)$ to $L^2_\alpha(\R_{+}^2, \mu^{-1}).$ 

By symmetry, noting $\nu= \left(\mu/\lambda \right)^{1/2}=(\lambda^{-1}/\mu^{-1})^{1/2}$ and also the fact that $\mu, \lambda \in \mathcal{B}_2$ implies $\mu^{-1}, \lambda^{-1} \in \mathcal{B}_2$, it is enough to prove the first assertion. Take $f \in L^2_\alpha(\mu).$ For convenience, split over the two dyadic grids:
\begin{align*}
\mathcal{A}_b f(z)& = \sum_{j=1,2}\sum_{I \in \calD^j} \left( \frac{1}{A_\alpha(Q_I)} \int_{Q_I} |b(w)-b_{Q_I}| |f(w)| \, dA_\alpha \right) \mathbf{1}_{Q_I}(z)\\
&:= \mathcal{A}_b^1 f(z)+\mathcal{A}_b^2 f(z).
\end{align*}
We estimate $\mathcal{A}_b^j$ using Lemma \ref{OscLemma}, the $b \in {\rm BMOA}_{\nu}\left(\R_{+}^2 \right)$ condition, and the doubling property of $\nu$ which guarantees $\nu(\widetilde{Q}_J) \lesssim \nu(Q_J) \lesssim \nu(Q_J^{up}) $ with implicit constants depending on $[\nu]_{\calB_2}$, which we do not track, as follows:
\begin{align*}
\mathcal{A}_b^j f(z)& = \sum_{I \in \calD^j} \left( \frac{1}{A_\alpha(Q_I)} \int_{Q_I} |b(w)-b_{Q_I}| |f(w)| \, dA_\alpha(w) \right) \mathbf{1}_{Q_I}(z)\\
& \lesssim \sum_{I \in \calD^j} \left( \frac{1}{A_\alpha(Q_I)} \sum_{\substack{J \in \calD^j, J \subseteq I}} \frac{1}{A_\alpha(\widetilde{Q}_J)} \int_{\widetilde{Q}_J} |b(w)-b_{Q_J}| \, dA_\alpha(w) \cdot   \int_{Q_J} |f(w)| \, dA_\alpha(w) \right) \mathbf{1}_{Q_I}(z) \\
& \lesssim \|b\|_{{\rm BMOA}_{\nu}\left(\R_{+}^2 \right)} \cdot \sum_{I \in \calD^j} \left( \frac{1}{A_\alpha(Q_I)} \sum_{\substack{J \in \calD^j, J \subseteq I}}   \frac{\nu(Q_J) }{A_\alpha(Q_J)} \int_{Q_J} |f(w)| \, dA_\alpha(w) \right) \mathbf{1}_{Q_I}(z) \\
& \lesssim \|b\|_{{\rm BMOA}_{\nu}\left(\R_{+}^2 \right)} \cdot \sum_{I \in \calD^j} \left( \frac{1}{A_\alpha(Q_I)} \sum_{\substack{J \in \calD^j, J \subseteq I}}   \frac{\nu(Q_J^{up}) }{A_\alpha(Q_J)} \int_{Q_J} |f(w)| \, dA_\alpha(w) \right)  \mathbf{1}_{Q_I}(z) \\
& \leq  \|b\|_{{\rm BMOA}_{\nu}\left(\R_{+}^2 \right)} \cdot  \sum_{I \in \calD^j} \left( \frac{1}{A_\alpha(Q_I)} \int_{Q_I} M_{\calD^j}(f)(w) \nu(w) \, dA_\alpha(w)  \right) \mathbf{1}_{Q_I}(z) \\
& =  \|b\|_{{\rm BMOA}_{\nu}\left(\R_{+}^2 \right)} \cdot \mathcal{A}_{\calD^j}\left(M_{\calD^j}(f) \nu\right)(z).
\end{align*}

Therefore, we can use this pointwise control to estimate the Lebesgue norm as follows, using the standard weighted bounds for the dyadic operators $\mathcal{A}_{\calD^j}$ and $M_{\calD^j}$ (see, e.g., \cite[Theorem 1]{RTW2017} or \cite[Theorem 4.3]{GHK2022}):
\begin{align*}
\|\mathcal{A}_b^j f\|_{L^2(\lambda)} & \lesssim \|b\|_{{\rm BMOA}_{\nu}\left(\R_{+}^2 \right)} \cdot\| \mathcal{A}_{\calD^j}\left(M_{\calD^j}(f) \nu \right)\|_{L^2_\alpha(\R_{+}^2, \lambda)}\\
& \lesssim  \|b\|_{{\rm BMOA}_{\nu}\left(\R_{+}^2 \right)} \cdot\| M_{\calD^j}(f) \nu\|_{L^2_\alpha(\R_{+}^2, \lambda)} \\
& =   \|b\|_{{\rm BMOA}_{\nu}\left(\R_{+}^2 \right)} \cdot\| M_{\calD^j} f \|_{L^2_\alpha(\R_{+}^2, \mu)} \\
&\lesssim   \|b\|_{{\rm BMOA}_{\nu}\left(\R_{+}^2 \right)} \cdot\| f \|_{L^2_\alpha(\R_{+}^2, \mu)},
\end{align*}
which completes the proof that $\mathcal{A}_b^j$ is bounded from $L^2_\alpha(\R_{+}^2, \mu)$ to $L^2_\alpha(\R_{+}^2, \lambda)$.  

To this end, note that the above argument also extends to a proof for the sufficient part of the boundedness of $[b, P_\alpha^+]$ under the two-weight setting, and we would like to leave the details to the interested reader. 

\subsection{Necessity} \label{20250116subsec02}
Suppose the commutator $[b, P_\alpha]$ is bounded from  $L_\alpha^2 \left(\R_{+}^2, \mu \right)$ to $L_\alpha^2 \left(\R_{+}^2, \lambda \right) $, that is, for every $f\in L_\alpha^2 \left(\R_{+}^2, \mu \right)$,
$$
\|[b, P_\alpha](f)\|_{L_\alpha^2 \left(\R_{+}^2, \lambda \right)}\lesssim \|f\|_{ L_\alpha^2 \left(\R_{+}^2, \mu \right)}.
$$
We now prove that  $b \in {\rm BMOA}_{\nu} \left(\R_{+}^2 \right)$. To do so, we will make use of the so-called \emph{complex median method} introduced in \cite{Wei2024}. The following lemma is \emph{key} to our argument. 

\begin{lem} \label{ComplexMedian}
Let $b$ be a complex measurable function on $\R_{+}^2$ and $Q$ be any cube on $\R_+^2$. Then there exist orthogonal lines $\ell_1, \ell_2 \in \mathbb{C}$ with intersection ${\bf b}_Q^{cm}$ that\footnote{Here, the upper script ``$cm$" in ${\bf b}_Q^{cm}$ is short for ``complex median".} divide $\mathbb{C}$ into four ``quadrants'' $T_Q^j$, $j=1,2,3,4$, placed in counterclockwise ordering so that the following estimate holds for $j \in \{1,2,3,4\}$:
$$A_\alpha(\{z \in Q: b(z) \in T_Q^j \}) \geq \frac{A_\alpha(Q_I)}{16}.$$
\end{lem}

To make use of the above lemma, we fix a Carleson tent $Q_I$ with $I \in \calD^j$, and choose the test set to be $S_I$, where
 $$ S_I :=\left\{z=x+iy \in \R_{+}^2: x \in I, y \in (\frakA |I|, (\frakA+1) |I|) \right\}. $$
 Here, $\frakA$ is a large positive constant to be determined later. The following facts are easy to prove and useful in the sequel: for $z \in Q_I$ and $w \in S_I$, we have
 \begin{enumerate} 
 \item [$\bullet$] the kernel estimate:
\begin{equation}  \label{kernelestimate} 
\frac{1}{|z-\overline{w}|^{2+\alpha}} \simeq_{\frakA} \frac{1}{A_\alpha(Q_I)} \simeq_{\frakA} \frac{1}{A_\alpha(S_I)}. 
\end{equation}
\item [$\bullet$] the angle estimate:
$$ 
\pi- \frac{\pi}{100}\leq   \arg(z-\overline{w}) \leq \pi + \frac{\pi}{100}.
$$
\end{enumerate}
Finally, for $\widetilde{I}:=2\frakA I$, we let $Q_{\widetilde{I}}$ be a large Carleson box that contains both $Q_I$ and $S_I$ (see, Figure \ref{Fig1} below). 

\vspace{0.1cm}

\begin{center}
\begin{figure}[ht]
\begin{tikzpicture}[scale=4]
\draw (-.5, 0) [->] -- (1.35, 0);
\fill [opacity=2] (0.25, 0) circle [radius=.2pt];
\fill [opacity=2] (0.55, 0) circle [radius=.2pt];
\draw (0.25, 0) -- (0.25, 0.3) -- (0.55, 0.3) -- (0.55, 0);
\draw (-0.1, 0) -- (-0.1, 0.9) -- (0.9, 0.9) -- (0.9, 0);
\draw (0.25, 0.5) -- (0.25, 0.8) -- (0.55, 0.8) -- (0.55, 0.5)--(0.25, 0.5);
\draw (0.25, -0.5) -- (0.25, -0.8) -- (0.55, -0.8) -- (0.55, -0.5)--(0.25, -0.5);
\fill (0.34, 0) node [below] {\small{$I$}}; 
\fill (0.8, 0) node [below] {\small{$\widetilde{I}$}};
\fill (1.3, 0) node [above] {\small{$\R_+^2$}}; 
\fill [opacity=2, red] (0.45, 0.2) circle [radius=.3pt];
\fill (0.49, 0.16) node [below] {\small{$z$}};
\fill [opacity=2, blue] (0.35, 0.7) circle [radius=.3pt];
\fill (0.35, 0.69) node [right] {\small{$w$}};
\fill [opacity=2, blue] (0.35, -0.7) circle [radius=.3pt];
\fill (0.35, -0.705) node [right] {\small{$\overline{w}$}};
\fill (0.56, 0.2) node [right] {\small{$Q_I$}};
\fill (0.56, 0.7) node [right] {\small{$S_I$}};
\fill (0.91, 0.8) node [right] {\small{$Q_{\widetilde{I}}$}};
\fill (0.56, -0.7) node [right] {\small{$\overline{S_I}$}};
\draw [dashed] (0.45, 0.19) -- (0.35, -0.7); 
\end{tikzpicture}
\caption{Construction of $S_I$ and $Q_{\widetilde{I}}$}
\label{Fig1}
\end{figure}
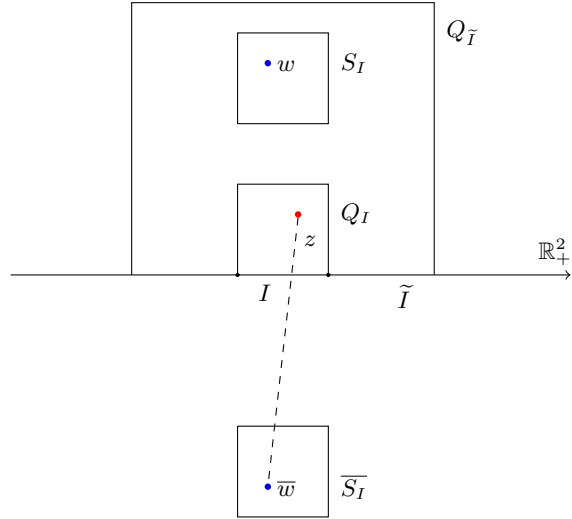
\end{center}

\vspace{-0.1in}

With these, we apply Lemma \ref{ComplexMedian} for the complex value function $b$ and $Q=S_I$. As a consequence, there exists four quadrants $T_{S_I}^j, j=1,2,3,4$ that partition $\C$,  such that 
\begin{equation} \label{20250119eq01}
A_\alpha(F_j)\geq \frac{A_\alpha(S_I)}{16}, 
\end{equation} 
where $F_j:=\{z \in S_I: b(z) \in T_{S_I}^j$\}. Moreover, let
\begin{equation} \label{20250119eq02}
B_j:=\{z\in  Q_I: b(z)  \in T_{S_I}^{j+2}\}, \quad j=1,2,3,4,
\end{equation}
where the integer $j+2$ is understood to be computed modulo $4$.
Note we may not have a uniform lower bound on $A_\alpha(B_j)$ for all $j$, but this will not matter. To this end, let further ${\bf b}_{S_I}^{cm} \in \mathbb{C}$ be the intersection of the two lines $\ell_1, \ell_2$ in Lemma \ref{ComplexMedian}, and $\theta=\theta(I)$ be the angle such that $e^{i\theta}$ rotates $\ell_1$ to $x$-axis, and $\ell_2$ to $y$-axis, respectively. This in particular gives for $z\in B_j$ and $w\in F_j$,
\begin{align} \label{20250116eq01}
    |b(z)-{\bf b}_{S_I}^{cm}|&\leq  |e^{i\theta}(b(z)-{\bf b}_{S_I}^{cm})|+|e^{i\theta}({\bf b}_{S_I}^{cm}-b(w))| \nonumber \\
    &\leq 2|e^{i\theta}(b(z)-b(w))|=2|b(z)-b(w)|.
\end{align}
Then, since $Q_I= \bigcup_{j=1}^4 B_j$, there holds 
\begin{align} \label{20250119eq23}
 &{1\over \nu(Q_I)}\int _{Q_I} |b(z)-b_{Q_I}| dA_\alpha(z) \nonumber \\
 &\leq  {2\over \nu(Q_I)}\int _{Q_I} |b(z)-{\bf b}_{S_I}^{cm}| dA_\alpha(z) \nonumber \\
 &= 2 \sum_{j=1}^{4} {1\over \nu(Q_I)}\int _{B_j} |b(z)-{\bf b}_{S_I}^{cm}| dA_\alpha(z).
\end{align}
We only need to consider the first term $j=1$, as the other terms are controlled similarly. Note that, due to \eqref{20250116eq01} and \eqref{kernelestimate}, we have
\begin{align} \label{20250116eq20}
 &  {1\over \nu(Q_I)}\int _{B_1} |b(z)-{\bf b}_{S_I}^{cm}| dA_\alpha(z) \nonumber \\
 &={1\over \nu(Q_I)}{1\over A_\alpha(F_1)}\int _{B_1}\int _{F_1} |b(z)-{\bf b}_{S_I}^{cm}|dA_\alpha(w) dA_\alpha(z) \nonumber \\
 &\leq{1\over\nu(Q_I)}{1\over A_\alpha(F_1)}\int _{B_1}\int _{F_1} |b(z)-b(w)|dA_\alpha(w) dA_\alpha(z) \nonumber \\
 &\lesssim{1\over \nu(Q_I)}\int _{B_1}\int _{F_1}{ |b(z)-b(w)|\over |z-\overline{w}|^{2+\alpha} }dA_\alpha(w) dA_\alpha(z).
\end{align}
Now the idea is to write the double integral in \eqref{20250116eq20} as an integral for the commutator $[b, P_\alpha]$. For this purpose, we need some technical steps to remove the absolute values in \eqref{20250116eq20}. More precisely, 

\medskip 

\noindent \underline{\textit{Step I: removing the absolute value in $|b(z)-b(w)|$.}}

\medskip

Recall that by the definition of $\theta=\theta(I)$, for any $z \in B_1$ and $w \in F_1$, 
$$
e^{i\theta} b(z) \in \left\{z=x+iy: x, y \ge 0\right\}  \quad \textrm{and} \quad e^{i\theta} b(w)\in \left\{z=x+iy: x, y \le 0\right\}, 
$$
which implies $\textrm{angle}(e^{i\theta}(b(z)-b(w))) \in \left[0, \frac{\pi}{2} \right]$, and hence
\begin{equation} \label{20250116eq30}
|b(z)-b(w)|, \; \im \left( e^{i\left(\theta-\frac{\pi}{4} \right)}(b(z)-b(w)) \right) \lesssim \re \left( e^{i\left(\theta-\frac{\pi}{4} \right)}(b(z)-b(w)) \right). 
\end{equation} 

\medskip 

\noindent \underline{\textit{Step II: removing the absolute value in $\frac{1}{|z-\overline{w}|^{2+\alpha}}$.}}

\medskip

Denote $z_{Q_I}$ and $w_{S_I}$ be the center of $Q_I$ and $S_I$, respectively, and let $\theta_1\in[0,2\pi)$ such that ${e^{i\theta_1} \over (z_{Q_I}-\overline{w_{S_I}})^{2+\alpha}}$ is positive, i.e.,
$$ 
{1 \over |z_{Q_I}-\overline{w_{S_I}}|^{2+\alpha}} = {e^{i\theta_1} \over (z_{Q_I}-\overline{w_{S_I}})^{2+\alpha}}. 
$$
Therefore, for every $z\in B_1$ and $w\in F_1$,
\begin{align*}
    \bigg|{e^{i\theta_1} \over (z-\overline{w})^{2+\alpha}}-{e^{i\theta_1} \over (z_{Q_I}-\overline{w_{S_I}})^{2+\alpha}}\bigg|
    & =\bigg|{1 \over (z-\overline{w})^{2+\alpha}}-{1 \over (z_{Q_I}-\overline{w_{S_I}})^{2+\alpha}}\bigg|\\
    &\lesssim {|z-z_{Q_I}| \over |z_{Q_I}-\overline{w_{S_I}}|^{3+\alpha}}+
    {|w-w_{S_I}| \over |z_{Q_I}-\overline{w_{S_I}}|^{3+\alpha}}\lesssim
    { |I| \over (\frakA|I|)^{3+\alpha}}\\
    &\lesssim { 1 \over \frakA}{\frac{1}{|z_{Q_I}-\overline{w_{S_I}}|^{2+\alpha}}}\\
    &={1\over \frakA}{e^{i\theta_1} \over (z_{Q_I}-\overline{w_{S_I}})^{2+\alpha}}, 
\end{align*}
which implies that 
\begin{align*}
    \bigg(1-{1\over \frakA}\bigg){e^{i\theta_1} \over (z_{Q_I}-\overline{w_{S_I}})^{2+\alpha}}\lesssim  \re\bigg({e^{i\theta_1} \over (z-\overline{w})^{2+\alpha}}\bigg)\lesssim \bigg(1+{1\over \frakA}\bigg){e^{i\theta_1} \over (z_{Q_I}-\overline{w_{S_I}})^{2+\alpha}}
\end{align*}
and
\begin{align*}
     \bigg|\im\bigg({e^{i\theta_1} \over (z-\overline{w})^{2+\alpha}}\bigg)\bigg|
    &\lesssim {1\over \frakA} \cdot  {e^{i\theta_1} \over (z_{Q_I}-\overline{w_{S_I}})^{2+\alpha}}.
\end{align*}
By choosing $\mathcal A$ large enough, we see that for every $z\in S_1$ and $w\in F_1$, there holds
\begin{equation} \label{20250116eq32}
\bigg|\im\bigg({e^{i\theta_1} \over (z-\overline{w})^{2+\alpha}}\bigg)\bigg|\leq \frac{2}{\frakA}  \cdot \re\bigg({e^{i\theta_1} \over (z-\overline{w})^{2+\alpha}}\bigg)
\end{equation} 
and 
\begin{equation} \label{20250116eq31}
    \bigg|{e^{i\theta_1} \over (z-\overline{w})^{2+\alpha}}\bigg|\lesssim  \re\bigg({e^{i\theta_1} \over (z-\overline{w})^{2+\alpha}}\bigg).
\end{equation} 

\medskip 

\noindent Therefore, by \eqref{20250116eq30}, \eqref{20250116eq31}, and \eqref{20250116eq20}, we have 
\begin{align} \label{20250119eq25}
& {1\over \nu(Q_I)}\int _{B_1} |b(z)-{\bf b}_{S_I}^{cm}| dA_\alpha(z) \lesssim {1\over \nu(Q_I)}\int _{B_1}\int _{F_1}{ |b(z)-b(w)|\over |z-\overline{w}|^{2+\alpha} }dA_\alpha(w) dA_\alpha(z) \nonumber \\
& \lesssim  {1\over \nu(Q_I)}\int _{B_1}\int _{F_1} \re \left( e^{i\left(\theta-\frac{\pi}{4} \right)}(b(z)-b(w))\right) \cdot \re\bigg( {e^{i\theta_1}\over (z-\overline{w})^{2+\alpha} }\bigg)dA_\alpha(w) dA_\alpha(z) \nonumber \\
&\lesssim {1\over \nu(Q_I)}\int _{B_1}\int _{F_1} \re\Bigg( \left( e^{i\left(\theta-\frac{\pi}{4} \right)}(b(z)-b(w)) \right) \cdot\bigg( {e^{i\theta_1}\over (z-\overline{w})^{2+\alpha} }\bigg)\Bigg)dA_\alpha(w) dA_\alpha(z) \nonumber \\
&\lesssim {1\over \nu(Q_I)} \int _{B_1}\bigg|\int _{F_1}  e^{i\left(\theta-\frac{\pi}{4} \right)}(b(z)-b(w)) \cdot\bigg( {e^{i\theta_1}\over (z-\overline{w})^{2+\alpha} }\bigg) dA_\alpha(w)\bigg| dA_\alpha(z) \nonumber\\
&= {1\over \nu(Q_I)} \int _{B_1}\bigg|\int _{F_1} ( b(z)-b(w))\cdot\bigg( {1\over (z-\overline{w})^{2+\alpha} }\bigg) dA_\alpha(w)\bigg| dA_\alpha(z) \nonumber \\
&={1\over \nu(Q_I)}\int _{B_1}\bigg|\int _{\R_+^2} ( b(z)-b(w))\cdot\bigg( {1\over (z-\overline{w})^{2+\alpha} }\bigg)\chi_{F_1}(w)dA_\alpha(w)\bigg| dA_\alpha(z) \nonumber \\
&= {1\over \nu(Q_I)}\int _{B_1}\bigg| [b,P_\alpha](\chi_{F_1})(z)\bigg| dA_\alpha(z)\\
 &={1\over \nu(Q_I)}\int _{B_1}\bigg| [b,P_\alpha](\chi_{F_1})(z)\bigg| \lambda^{1\over2}(z)\lambda^{-{1\over2}}(z)dA_\alpha(z) \nonumber\\
 &\leq {1\over \nu(Q_I)}\Bigg(\int _{B_1}\bigg| [b,P_\alpha](\chi_{F_1})(z)\bigg|^2 \lambda(z)dA_\alpha(z)\Bigg)^{1\over2} 
 \Bigg( \int _{B_1}\lambda^{-1}(z)dA_\alpha(z)\Bigg)^{1\over2} \nonumber \\
 &\lesssim {1\over \nu(Q_I)} \cdot \|[b, P_\alpha]\|_{L_\alpha^2 \left(\R_{+}^2, \mu \right) \mapsto L_\alpha^2 \left(\R_{+}^2, \lambda \right) } \Big\|\chi_{F_1} \Big\|_{L^2_\alpha(\R_+^2, \mu)}
\big(\lambda^{-1}(B_1)\big)^{1\over2} \nonumber \\
 &= {1\over \nu(Q_I)} \cdot \mu(F_1)^{1\over2}
\big(\lambda^{-1}(B_1)\big)^{1\over2}  \|[b, P_\alpha]\|_{L_\alpha^2 \left(\R_{+}^2, \mu \right) \mapsto L_\alpha^2 \left(\R_{+}^2, \lambda \right) } \nonumber \\
& \lesssim {1\over \nu(Q_{\widetilde{I}})} \cdot \mu(Q_{\widetilde{I}})^{1\over2}
\big(\lambda^{-1}(Q_{\widetilde{I}})\big)^{1\over2}  \|[b, P_\alpha]\|_{L_\alpha^2 \left(\R_{+}^2, \mu \right) \mapsto L_\alpha^2 \left(\R_{+}^2, \lambda \right) } \nonumber \\
&\simeq \|[b, P_\alpha]\|_{L_\alpha^2 \left(\R_{+}^2, \mu \right) \mapsto L_\alpha^2 \left(\R_{+}^2, \lambda \right) }. \nonumber
\end{align}
Here, in the second estimate above, we have used the fact that 
$$
\im \left( e^{i\left(\theta-\frac{\pi}{4} \right)}(b(z)-b(w) \right) \cdot \im \bigg( {e^{i\theta_1}\over (z-\overline{w})^{2+\alpha} }\bigg)  \le \frac{2}{\mathcal A} \re \left( e^{i\left(\theta-\frac{\pi}{4} \right)}(b(z)-b(w) \right) \cdot \re \bigg( {e^{i\theta_1}\over (z-\overline{w})^{2+\alpha} }\bigg)
$$
as $\arg  e^{i\left(\theta-\frac{\pi}{4} \right)}(b(z)-b(w)) \in \left[-\frac{\pi}{4}, \frac{\pi}{4} \right]$;  while in the second last estimate, we used the estimate: for any $I \in \calD$, 
\begin{equation} \label{20250117eq10}
 {1\over \nu(Q_I)} \cdot \left[\mu(Q_I)\right]^{1\over2}
\left[\lambda^{-1}(Q_I)\right]^{1\over2} \simeq 1. 
\end{equation} 
Indeed, recall that $\nu:=\mu^{\frac{1}{2}}\lambda^{-\frac{1}{2}}$. On one hand side, 
$$
\nu(Q_I)=\int_{Q_I} \nu(z)dA_\alpha(z)=\int_{Q_I} \mu^{\frac{1}{2}}(z)\lambda^{-\frac{1}{2}}(z)dA_\alpha(z) \le   \left[\mu\left(Q_I \right)\right]^{\frac{1}{2}} \left[\lambda^{-1} \left(Q_I\right) \right]^{\frac{1}{2}}, 
$$
which implies 
$$
{1\over \nu(Q_I)} \cdot \left[\mu(Q_I)\right]^{1\over2}
\left[\lambda^{-1}(Q_I)\right]^{1\over2} \gtrsim 1.
$$
On the other hand side, since $\nu \in \calB_2$, there holds
$$
\frac{\nu(Q_I)\nu^{-1}(Q_I)}{A^2_\alpha(Q_I)} \simeq 1,
$$
and hence by Cauchy-Schwarz, 
$$
\frac{1}{\nu (Q_I)} \simeq \frac{\nu^{-1}(Q_I)}{A_\alpha^2 (Q_I)} \le \frac{\left[\mu^{-1}\left(Q_I \right)\right]^{\frac{1}{2}} \left[\lambda \left(Q_I\right) \right]^{\frac{1}{2}}}{A_\alpha^2 (Q_I)} \simeq \frac{1}{\left[\mu\left(Q_I \right) \right]^{\frac{1}{2}} \left[\lambda^{-1} \left(Q_I\right) \right]^{\frac{1}{2}}},
$$
and this gives the second estimate 
$$
{1\over \nu(Q_I)} \cdot \left[\mu(Q_I) \right]^{1\over2}
\left[\lambda^{-1}(Q_I)\right]^{1\over2} \lesssim 1.
$$
Hence, the estimate \eqref{20250117eq10} is proved. 

\vspace{0.1in}

Finally, arguing similarly, one can verify that for $j=2, 3, 4$, 
\begin{align*}
 {1\over \nu(Q_I)}\int _{B_j} |b(z)-{\bf b}_{S_I}^{cm}| dA_\alpha(z) &\lesssim \|[b, P_\alpha]\|_{L_\alpha^2 \left(\R_{+}^2, \mu \right) \mapsto L_\alpha^2 \left(\R_{+}^2, \lambda \right)},
\end{align*}
which concludes
\begin{align*}
{1\over \nu(Q_I)}\int _{Q_I} |b(z)-b_{Q_I}| dA_\alpha(z)&\lesssim \|[b, P_\alpha]\|_{L_\alpha^2 \left(\R_{+}^2, \mu \right) \mapsto L_\alpha^2 \left(\R_{+}^2, \lambda \right) }.
\end{align*}

The proof is complete for $[b, P_\alpha]$. To this end, we turn to the necessary part for the boundedness of the commutator $[b, P_\alpha^{+}]$, which is indeed simpler and follows line by line from the above argument. The only modification, since the kernel of $P_\alpha^{+}$ is now $\frac{1}{|z-\overline{w}|^{\alpha+2}}$, is that one does \emph{not} need to remove the absolute values in the kernel (see, \textit{Step II} above) and can prove
$$
{1\over \nu(Q_I)}\int_{B_1}\int _{F_1}{ |b(z)-b(w)|\over |z-\overline{w}|^{2+\alpha} }dA_\alpha(w) dA_\alpha(z) \lesssim {1\over \nu(Q_I)}\int _{B_1}\bigg| [b,P^+_\alpha](\chi_{F_1})(z)\bigg| dA_\alpha(z)
$$
simply via \eqref{20250116eq30}. Hence, we omit the details here.

\section{Proof of Theorem \ref{mainthm002}}\label{s:3}

Here, we only prove Theorem \ref{mainthm002} for the commutator $[b, P_\alpha]$; while a straightforward modification will assert the compactness results for $[b, P_\alpha^{+}]$, and hence we would like to leave the details to the interested reader. We divide our proofs into two parts.  

\vspace{-0.1in}

\subsection{Sufficiency}
This is seen as follows.  Suppose $ b \in \textnormal{VMOA}_{\nu }(\mathbb R_{+}^2)$ with $\|b\|_{\textnormal{BMOA}_{\nu }(\R_+^2)}=1$. 
We will show that for any fixed $ 0< \epsilon < 1$, we have the decomposition 
\begin{equation} \label{20250118eq01} 
[b,P_\alpha] =  A_\epsilon+B_\epsilon,
\end{equation} 
where the operator norm of $A_\epsilon$ acting boundedly from $L_\alpha^2 \left(\R_{+}^2, \mu \right)$ to $L_\alpha^2 \left(\R_{+}^2, \lambda \right) $ is at most $ \epsilon$, and $ B$ is compact. Therefore, $[b,P_\alpha]$ is a limit in the operator norm of compact operators acting from   $L_\alpha^2 \left(\R_{+}^2, \mu \right)$ to $L_\alpha^2 \left(\R_{+}^2, \lambda \right) $, which implies $[b,P_\alpha]$ is itself compact.

We will follow an argument similar to \cite{LaceyLi2022}. Let $\eta>0$ be small. Let $\phi:[0, \infty) \rightarrow [0,1]$ be a compactly supported, smooth bump function supported on $[0, \eta]$ and satisfying $\phi \equiv 1$ on $[0, \frac{\eta}{2}].$ Recall that the operator $P_\alpha$ has integral kernel $K_\alpha(z,w)=  \frac{i^{2+\alpha}}{(z-\overline{w})^2}$, which we will now decompose in an appropriate way to reflect the VMO structure. Set $ K_\alpha = \sum_{t=0} ^{3} K_\alpha^t$, where  
\begin{gather*}
 K_\alpha^0 (z,w)  = 
\phi(|z-\overline{w}|) \left(1-\phi\left(\frac{1}{|z|+|w|}\right)\right) K_\alpha(z,w);
\\
 K_\alpha^1 (z,w)  =  
\phi(|z-\overline{w}|)\, \phi\left(\frac{1}{|z|+|w|}\right) K_\alpha(z,w);
 \\
  K_\alpha^2(z,w)  =  
(1-\phi(|z-\overline{w}|)\,\phi\left(\frac{1}{|z|+|w|}\right) K_\alpha(z,w);
 \\
K_\alpha^3(z,w)= K_\alpha(z,w)- \sum_{t=0}^2 K_\alpha^t(z,w)= (1-\phi(|z-\overline{w}|))\left(1-\phi\left(\frac{1}{|z|+|w|}\right)\right) K_\alpha(z,w).
\end{gather*}
Write $ P_\alpha^t$ for the operator associated to the kernel $ K_\alpha^t$, so that $$[b, P_\alpha]= \sum_{t=0}^{3} [b, P_\alpha^t].$$
We shall see that $[b, P_\alpha^3]$ and $\sum_{t=0}^2 [b, P_\alpha^t]$ will play the role of $A_\epsilon$ and $B_\epsilon$, respectively. 

\medskip 

\noindent Now, let us turn to the details. We first claim that 
\begin{equation*}
\left\| [b,P_\alpha^t] \right\| _{L^{2} (\R_+^2, \mu ) \mapsto L^{2} (\R_+^2, \lambda)} \lesssim \epsilon _ \eta , \qquad t=0,1,2,  
\end{equation*}
where $ \epsilon _ \eta $ decreases to zero as $ \eta $ does. For the cases $t=0,1$ notice that $K_\alpha^t(z,w) \neq 0$ only if $|z-\overline{w}| \leq \eta$, which geometrically means that the points $z$ and $w$ are both close to each other and close to the real axis. Therefore, we run the dyadic domination proof leading to the estimate in Corollary \ref{SparseCom}, it is rather easy to see we can achieve an estimate of the form 
\begin{align*} 
|[b,P_\alpha^t] f(z)|&  \lesssim  \sum_{j=1,2}\sum_{\substack{I \in \calD^j: \\ |I| < \delta_\eta}} \left( \frac{1}{A_\alpha(Q_I)} \int_{Q_I} |b-b_{Q_I}| |f| \, dA_\alpha \right) \mathbf{1}_{Q_I}(z) \\
& \quad  +  \sum_{j=1,2}\sum_{\substack{I \in \calD^j: \\ |I| < \delta_\eta}} |b(z)-b_{Q_I}| \left( \frac{1}{A_\alpha(Q_I)}  \int_{Q_I} |f| \, dA_\alpha \right) \mathbf{1}_{Q_I}(z),
\end{align*}
where $\delta_\eta$ decreases to $0$ as $\eta \rightarrow 0.$ By the ${ \rm VMOA}_\nu(\R_+^2)$ condition, if $\delta_\eta$ is small enough, we can guarantee ${1\over \nu(Q_I)}\int _{Q_I} |b(z)-b_{Q_I}|dA_\alpha(z)$ is as small as we wish for all $|I|\leq \delta_\eta$. Then by applying the proof of the sufficient part of Theorem \ref{mainthm001}, we can achieve 
$$ \lVert [b,P_\alpha^t] \rVert _{L ^{2} (\mu, \mathbb{R}_{+}^2 ) \mapsto L ^{2} (\lambda, \mathbb{R}_{+}^2 )} \lesssim \epsilon _ \eta.$$
A very similar argument works for $[b, P_\alpha^2]$, except this time the summation in the positive dyadic operator will only involve Carleson tents whose centers are far from the origin (this includes ``large'' Carleson tents as well). 

It remains for us to argue that  $  [b,P_\alpha^3]   $ is a compact operator.  First, notice that the integral kernel associated with $[b,P_\alpha^3]$ is
$$
C_3(z,w):=\left[b(z)-b(w) \right] K_\alpha^3(z,w).
$$ 
We remark that it is easy to see that $C_3(z,w)$ is compactly supported on $\mathcal{K}_\eta:= \{(z,w) \in \R_{+}^2 \times \R_{+}^2:|z|, |w|\leq  \frac{2}{\eta}\}$ and that there exists $c_\eta>0$ so that
$$ 
|C_3(z,w)| \leq c_\eta |b(z)-b(w)|, \quad z,w \in \mathcal{K}_\eta.
$$
Indeed, $K_\alpha^3(z,w)$ is supported on the smaller set $\mathcal{K}_\eta \cap \{(z,w) \in \R_{+}^2 \times \R_{+}^2:|z-\overline{w}|> \eta/2\}.$ Note that there exists an interval $I_\eta \in \calD_1 \cup \calD_2$ so that $\mathcal{K}_\eta \subseteq Q_{I_\eta}$ with comparable size (this choice of interval is certainly not unique, and depends on $\eta$). 

Now we can argue that $[b, P_\alpha^3]$ is compact. Note that the operator
$[b, P_\alpha^3]$ is compact as an operator from $L^2(\mu, \R_{+}^2)$  to $ L^2(\lambda,\R_{+}^2)$ if and only if
$ \lambda^{1\over 2}[b,P_\alpha^3] \mu^{-{1\over 2}}$ is compact as an operator from $L^2_\alpha(\R_{+}^2)$ to itself. We will now argue that $ \lambda^{1\over 2}[b,P_\alpha^3] \mu^{-{1\over 2}}$ is Hilbert--Schmidt using the integral kernel characterization, which of course will imply compactness. We estimate:
\begin{align*}
& \int_{\R_{+}^2} \int_{\R_{+}^2} \lambda(z) |C_3(z,w)|^2 \mu^{-1}(w) \, dA_\alpha(w) dA_\alpha(z)\leq c_{\eta}^2 \int_{Q_{I_\eta}} \int_{Q_{I_\eta}} \lambda(z) |b(z)-b(w)|^2 \mu^{-1}(w) \, dA_\alpha(w) dA_\alpha(z) \\
& \lesssim \mu^{-1}(Q_\eta) \int_{Q_{I_\eta}}  \lambda(z) |b(z)-b_{Q_\eta}|^2  \, dA_\alpha(z) + \lambda(Q_\eta) \int_{Q_{I_\eta}} |b(w)-b_{Q_\eta}|^2 \mu^{-1}(w)  \, dA_\alpha(w).
\end{align*}
Now, using Lemma \ref{OscLemma}, we have, assuming $I_\eta \in \calD^j$:

\begin{align*}
\int_{Q_{I_\eta}}  \lambda(z) |b(z)-b_{Q_\eta}|^2  \, dA_\alpha(z) & \lesssim \int_{Q_{I_\eta}} \lambda(z)
\left| \sum_{J \in \calD^j, \; J \subseteq I_\eta} \left( \frac{1}{A_\alpha(Q_J)} \int_{Q_J} |b-b_{Q_J}| \, d A_\alpha \right) \mathbf{1}_{Q_J}(z) \right|^2 dA_\alpha(z) \\
& \leq \int_{Q_{I_\eta}} \lambda(z)
\left| \sum_{J \in \calD^j, \; J \subseteq I_\eta} \frac{\nu(Q_J)}{A_\alpha(Q_J)} \mathbf{1}_{Q_J}(z) \right|^2 dA_\alpha(z).
\end{align*}
By a localized version of the standard weighted estimate for the sparse form, the above display can be controlled by a constant times
$$ 
\int_{Q_{I_\eta}} \nu^2 \lambda \, dA_\alpha= \mu(Q_{I_\eta}).
$$ 
A similar argument for the other term gives 
$$ 
\int_{Q_{I_\eta}} |b(w)-b_{Q_\eta}|^2 \mu^{-1}(w)  \, dA_\alpha(w) \lesssim \lambda^{-1}(Q_{I_\eta}).
$$
Therefore, we have shown:
\begin{align*}
\int_{\R_{+}^2} \int_{\R_{+}^2} \lambda(z) |C_2(z,w)|^2 \mu^{-1}(w) \, dA_\alpha(w) dA_\alpha(z)  & \lesssim \mu(Q_{I_\eta}) \mu^{-1}(Q_{I_\eta})+ \lambda(Q_{I_\eta}) \lambda^{-1}(Q_{I_\eta})\\
& \leq ([\mu]_{\calB_2}+ [\lambda]_{\calB_2}) A_\alpha(Q_{I_\eta})^2\\[3pt]
&< \infty, 
\end{align*}
which proves the Hilbert--Schmidt claim.

\subsection{Necessity}

We prove this by contradiction. Hence, there exists some $b\in \textrm{BMOA}_{\nu }^\alpha(\mathbb R_{+}^2) \backslash \textrm{VMOA} _{\nu }^\alpha(\mathbb R_{+}^2)$ such that $[b,P_\alpha]$ is compact from $L_\alpha^2 \left(\R_{+}^2, \mu \right)$ to $L_\alpha^2 \left(\R_{+}^2, \lambda \right)$.

Since $b\not\in \textnormal{VMOA}_{\nu }^\alpha(\mathbb R_{+}^2)$, then at least one of the two conditions in Definition \ref{def-VMOA} does not hold.  
We present the case where the first condition in Definition~\ref{def-VMOA} fails, that is, there exist $\delta_0>0$ and a sequence of Carleson boxes $\{Q_{I_j}\}_{j=1}^\infty=\{Q_{I_j}(c_{I_j},|I_j|)\}_{j=1}^\infty \subset \mathbb R^2_+$ such that $|I_j|\to 0$ as $j\to\infty$ and that
\begin{equation}\label{(i) not hold}
 {1\over \nu(Q_{I_j})}\int _{Q_{I_j}} |b(z)-b_{Q_{I_j}}|dA_\alpha(z)\geq\delta_0.
\end{equation}
Without loss of generality, we can further assume that 
\begin{align}\label{ratio1}
50|I_{j+1}|\leq  |I_{j}|.
\end{align}
Now, for each such fixed Carleson tent $Q_{I_j}$, we choose the other test set to be $S_{I_j}$ where
$$ 
S_{I_j} :=\left\{z=x+iy \in \R_{+}^2: x \in I_j, y \in (\frakA |I_j|, (\frakA+1) |I_j|) \right\} 
$$
and as before $\frakA$ is some large positive constant. Then, as in \eqref{kernelestimate}, one has for $z \in Q_{I_j}$ and $w \in S_{I_j}$,
$$ 
\frac{1}{|z-\overline{w}|^{2+\alpha}} \simeq_\frakA \frac{1}{A_\alpha(Q_{I_j})} \simeq_\frakA \frac{1}{A_\alpha(S_{I_j})}.
$$
Following the technique in the proof of the necessary part of Theorem \ref{mainthm001}, for each $Q_{I_j}$ and $S_{I_j}$, by choosing $\frakA$ large enough, we have the corresponding complex complex median ${\bf b}_{S_{I_j}}^{cm}$ and the sets $F_{j, k}$ and $B_{j, k}$ for $k=1,2,3,4$  (see, \eqref{20250116eq01}, \eqref{20250119eq01} and \eqref{20250119eq02}, respectively).

Now consider 
$$
\widetilde F_{j,k}:= F_{j,k}\backslash \bigcup_{\ell=j+1}^\infty S_{I_\ell}\qquad {\rm for}\ j=1,2,\ldots
$$
and for $k=1,2,3,4$. Then, by \eqref{ratio1}, we obtain that
for each $j$,
\begin{align}\label{Fj1}
A_\alpha(\widetilde F_{j,k}) &\geq A_\alpha(F_{j,k})- A_\alpha \left(\bigcup_{\ell=j+1}^\infty S_{I_\ell} \right) \geq
{1\over 16} A_\alpha(S_{I_j})-\sum_{\ell=j+1}^\infty A_\alpha \left(  S_{I_\ell} \right) \nonumber \\
& \geq {1\over 16} A_\alpha(S_{I_j})- {1\over 48}A_\alpha(S_{I_j})= {1\over 24} A_\alpha(S_{I_j}).
\end{align}

Now for each $j$, arguing as in \eqref{20250119eq23}, we have that
\begin{align*}
&{1\over \nu(Q_{I_j})}\int _{Q_{I_j}} |b(z)-b_{Q_{I_j}}|dA_\alpha(z)\\
& \leq  {2\over \nu(Q_{I_j})}\int _{Q_{I_j}} |b(z)-{\bf b}_{S_{I_j}}^{cm}| dA_\alpha(z)
\\
&\leq \sum_{k=1}^{4}
 {2\over \nu(Q_{I_j})}\int _{B_{j,k}} |b(z)-{\bf b}_{S_{I_j}}^{cm}| dA_\alpha(z), 
\end{align*}
which together with \eqref{(i) not hold} gives as least one of the following inequalities holds:
\begin{align*}
{2\over \nu(Q_{I_j})}\int _{B_{j,k}} |b(z)-{\bf b}_{S_{I_j}}^{cm}| dA_\alpha(z) \geq {\delta_0\over4},\qquad k=1,2,3,4.
\end{align*}
Without loss of generality, we now assume that $k=1$ holds, i.e., 
\begin{align*}
{2\over \nu(Q_{I_j})}\int _{B_{j,1}} |b(z)-{\bf b}_{S_{I_j}}^{cm}| dA_\alpha(z)  \geq {\delta_0\over4}.
\end{align*}
Therefore, for each $j$, following the argument in deriving \eqref{20250119eq25} with $F_1$ there replacing by $\widetilde{F}_{j, 1}$ (which is guaranteed by \eqref{Fj1}) and using \eqref{20250117eq10}, we obtain that 
\begin{align*}
{\delta_0\over4}&\leq {2\over \nu(Q_{I_j})}\int _{B_{j,1}} \left|b(z)-{\bf b}_{S_{I_j}}^{cm}\right| dA_\alpha(z) \\
&\lesssim 
 {1\over \nu(Q_{I_j})}\int _{B_{j,1}}\bigg| [b,P_\alpha](\chi_{\widetilde{F}_{j,1}})(z)\bigg| dA_\alpha(z)
 \\
&\lesssim {1\over \left[\mu(B_{j,1})\right]^{1\over 2}   \left[\lambda^{-1}(B_{j,1}) \right]^{1\over 2}}  \int_{B_{j,1}}\left|[b, P_\alpha](\chi_{\widetilde F_{j,1}})(z)\right| dA_\alpha(z)\\
&={1\over  \left[\lambda^{-1}(B_{j,1})\right]^{1\over 2}}  \int_{B_{j,1}}\left|[b, P_\alpha]\bigg({\chi_{\widetilde F_{j,1}} \over \left[\mu(B_{j,1})\right]^{1\over 2}}\bigg)(z)\right| dA_\alpha(z).
\end{align*}
Next, by H\"older's inequality, we further have
\begin{align}\label{uniform lower bound}
\delta_0 &\lesssim {1\over \left[\lambda^{-1}(B_{j,1})\right]^{1\over 2}}  \int_{B_{j,1}}\left|[b, P_\alpha]\big(f_j\big)(z)\right|  \lambda^{1\over 2}(z) \lambda^{-{1\over 2}}(z) dA_\alpha(z)\\
&\lesssim
{1\over  \left[\lambda^{-1}(B_{j,1}) \right]^{1\over 2}}  \left[\lambda^{-1}(B_{j,1})\right]^{1\over 2} \bigg( \int_{B_{j,1}}\big|[b, P_\alpha](f_j)(z)\big|^2 \lambda(z) dA_\alpha(z)\bigg)^{1\over 2}\nonumber\\
&\lesssim
\bigg( \int_{B_{j,1}}\big|[b, P_\alpha](f_j)(z)\big|^2 \lambda(z) dA_\alpha(z)\bigg)^{1\over 2},\nonumber
\end{align}
where in the above inequalities, we denote
$$ f_j(z) := {\chi_{\widetilde F_{j,1}}(z) \over \left[\mu(B_{j,1})\right]^{1\over 2}}.$$
This is a sequence of functions with disjoint supports, by \eqref{Fj1}, with $ \|f_j\|_{L_\alpha^2 \left(\R_{+}^2, \mu \right)} \simeq  1 $. 
Moreover, it is clear that this sequence $\{f_j\}_j$ is uniformly bounded in $L_\alpha^2 \left(\R_{+}^2, \mu \right)$.

Returning to the assumption of compactness of 
$[b,P_\alpha]$, 
the sequence  $  \{[b, P_\alpha](f_j)\}_j$ has a  
convergent subsequence, and without loss of generality, we assume that there is $\phi \in  L_\alpha^2 \left(\R_{+}^2, \lambda \right)$
such that 
$$
[b, P_\alpha](f_j) \rightarrow \phi \quad \textnormal{in} \ L_\alpha^2 \left(\R_{+}^2, \lambda \right).
$$
Moreover, from \eqref{uniform lower bound}, we see that 
$$
\|[b, P_\alpha](f_j)\|_{L_\alpha^2 \left(\R_{+}^2, \lambda \right)}\gtrsim \delta_0,
$$
which implies that 
 $ \lVert \phi \rVert _{L_\alpha^2 \left(\R_{+}^2, \lambda \right)} \gtrsim 1$.  We now choose a subsequence $\{ j_i\}_i$ so that 
\begin{equation} \label{20250120eq40} 
\|\phi- [b, P_\alpha](f_{j_i})\|_{L_\alpha^2 \left(\R_{+}^2, \lambda \right)} \leq  2^{-i}. 
\end{equation}
To get the desired contradiction, we let $ \psi:= \sum_{i} \frac{f_{j_i}}{i}$.
It is easy to see that 
\begin{align*}
    \|\psi\|_{L_\alpha^2 \left(\R_{+}^2, \mu \right)}&=
    \bigg(\int_{\mathbb R_+^2} \bigg|\sum_{i} \frac{f_{j_i}(z)}{i} \bigg|^2\mu(z)dA_\alpha(z)\bigg)^{1\over2}\\
    &=
    \bigg(\int_{\mathbb R_+^2} \sum_{i} \bigg| \frac{f_{j_i}(z)}{i}\bigg|^2\mu(z)dA_\alpha(z)\bigg)^{1\over2}\\
    &=\bigg(\sum_{i} \frac{1}{i^2} \int_{\mathbb R_+^2}  \big|f _{j_i}(z)\big|^2\mu(z)dA_\alpha(z)\bigg)^{1\over2}\\ &\lesssim 1, 
\end{align*}
which shows that $\psi$ is in $L_\alpha^2 \left(\R_{+}^2, \mu \right)$, and hence 
\begin{equation} \label{20250120eq11}
\left\| [b, P_\alpha] \psi \right\|_{L_\alpha^2 \left(\R_+^2, \lambda \right)} \lesssim 1. 
\end{equation} 
Here we use the facts that all $f_{j}$
have disjoint compact supports and that
$ \|f_j\|_{L_\alpha^2 \left(\R_{+}^2, \mu \right)} \simeq  1$. On the other hand side, by \eqref{20250120eq40}, one has
\begin{equation} \label{20250120eq12}
\bigg\| \sum_{i=1}^\infty \frac{1}{i} \bigl ( \phi - [b, P_\alpha](f_{j_i}\bigr) \bigg\|_{L_\alpha^2 \left(\R_{+}^2, \lambda \right)}
\lesssim  \left[   \sum_{i} \|\phi- [b, P_\alpha](f_{j_i}) \| _{L_\alpha^2 \left(\R_{+}^2, \lambda \right)} ^{2}  \right] ^{1/2} \lesssim 1. 
\end{equation} 
Therefore, combining \eqref{20250120eq11} and \eqref{20250120eq12} yields for any $N \in \N$
$$\phi_N:=\sum_{i=1}^N  \frac{\phi}{i}=[b, P_\alpha] \left( \sum_{i=1}^N \frac{f_{j_i}}{i} \right) +\sum_{i=1}^N \frac{1}{i} \left( \phi - [b, P_\alpha](f_{j_i}) \right) \in L_\alpha^2 \left(\R_+^2, \lambda \right),
$$  
with an uniform upper bound $\sup_{N \ge 1} \left\|\phi_N\right\|_{L_\alpha^2(\R_+^2, \lambda)} \lesssim 1$. However, it is clear that $\phi_N$ has a $L^2(\R_+^2, \lambda)$ norm at least of the magnitude $\log N$ since $ \lVert \phi \rVert _{L_\alpha^2 \left(\R_{+}^2, \lambda \right)} \gtrsim 1$, which gives the desired contradiction. 

\bigskip
\section{The Case of Non-Analytic Symbol} \label{s:4}
In this section, we will prove the boundedness theorem for the general (non-analytic) symbol $b$ under the one-weight setting. We will also provide an example to show that the $\mathcal{APR}$ condition on the weight $\sigma$ cannot be removed in this case, in contrast to the analytic case.

Since our methods will be slightly different from the approach in proving the Bloom-type result Theorem \ref{mainthm001}, we need to introduce some new function spaces. All these definitions can be found in \cite{Si2022} where they are given for the Siegel upper-half space, but they are even simpler for $\R_{+}^2.$ We recall $\beta(z,r)$ denotes a disk of center $z$ and radius $r$ in the Bergman metric and $d_\beta(z,w)$ denotes the Bergman distance between the points $z$ and $w.$

\begin{defn}
Let $b \in L^2_{\rm{loc}}(\R_{+}^2)$ and $r>0.$ We say 
\begin{enumerate}
    \item [$\bullet$] $b\in \textnormal{BMO}_r^2(\R_+^2)$  if
$$
\|b\|_{\textnormal{BMO}_r^2}:= \sup_{z \in \R_{+}^2} \left(\frac{1}{A(\beta(z,r))} \int_{\beta(z,r)} |b-  b_{\beta(z,r)}|^2 \, dA \right)^{1/2}< \infty. 
$$
\item [$\bullet$] $b \in \textnormal{BO}_r(\R_+^2)$  if $b$ is continuous and 
$$
\|b\|_{\textnormal{BO}_r}:= \sup_{z \in \R_{+}^2} \sup_{w \in \beta(z,r)} |b(z)-b(w)|< \infty. 
$$
\item [$\bullet$] $b \in \textnormal{BA}_r(\R_+^2)$  if
$$
\|b\|_{\textnormal{BA}_r}:= \sup_{z \in \R_{+}^2} \left(\frac{1}{A(\beta(z,r))} \int_{\beta(z,r)} |b|^2 \, dA \right)^{1/2}< \infty. 
$$
\end{enumerate}
\end{defn}

It is well-known that the spaces $\textnormal{BMO}_r^2(\R_{+}^2)$, $\textnormal{BO}_r(\R_{+}^2)$, and  $\textnormal{BA}_r(\R_{+}^2)$ are independent of the radius $r$. Therefore, we will omit the $r$ from the notation in the rest of the paper (take $r=1$ for simplicity). These spaces can also be alternatively characterized using the Berezin transform. Moreover, one can check that this definition of $\textnormal{BMO}^2_r(\R_+^2)$ agrees with Definition \ref{def-BMO} for any fixed $r$; that is $\textnormal{BMO}^2_r(\R_{+}^2)= \textnormal{BMO}^2(\R_{+}^2)$ for all $r$ with equivalent norms. We refer the reader, for example, \cite[Section 5]{HHLPW2024} and \cite{Si2022}, for more detailed information about the above claims. 

\vspace{0.1cm}

The following propositions can be found in \cite{Si2022}:

\begin{prop}\label{splitting}
Any function $b \in \textnormal{BMO}^2(\R_{+}^2)$ has a decomposition $b=b_1+b_2$ with $b_1 \in \textnormal{BO}(\R_{+}^2)$, $b_2 \in \textnormal{BA}(\R_{+}^2)$, and $\|b\|_{\textnormal{BMO}^2} \simeq\|b_1\|_{\textnormal{BO}} + \|b_2\|_{\textnormal{BA}}.$

\end{prop}

\begin{prop}\label{Lip}
Let $b \in \textnormal{BO}(\R_{+}^2)$. Then the following Lipschitz-type estimate holds for all $z,w \in \R_{+}^2$:
$$
|b(z)-b(w)|\leq \|b\|_{\textnormal{BO}} \cdot (1+d_\beta(z,w)).
$$
\end{prop}

The following useful estimate for the Bergman metric, valid for each $\varepsilon>0$, follows from a straightforward calculation (see, e.g., \cite{Si2022}):
$$
|d_\beta(z,w)| \leq C_{\varepsilon} \left(\frac {(\im z)^{-\varepsilon} (\im w)^{-\varepsilon}}{|z-\overline{w}|^{2-2 \varepsilon}} \right), \quad z, w \in \R_{+}^2.
$$

We will also need the following fact about a Whitney-type decomposition of the upper-half space, which is straightforward to prove (see, \cite[Theorem 2.5]{HH2021} for a more general statement of this kind of result).

\begin{prop} \label{diskslattice}
There exists an $R>0 $ and points $c_I \in Q_I^{up}$ so that $Q_I^{up} \subset \beta(c_I,R)$ for all intervals $I \in \mathcal{D}$, and moreover $A(Q_I^{up})| \sim A(\beta(c_I,R)).$ Moreover, the ``dyadic'' disks $\beta(c_{I},2R)$ have finite overlap for intervals in a fixed dyadic lattice, $I\in \calD_j.$ . 

\end{prop}

\begin{prop}\label{Carleson} Let $b \in \textnormal{BA}.$ Then $|b|^2 \, dA_\alpha$ is a Carleson measure for the Bergman space $A^2_\alpha(\R_+^2)$. 
\end{prop}

 \begin{proof}
 Take $f \in A^2_\alpha(\R_+^2)$ and fix a dyadic lattice $\calD$ on $\R$. It is easy to see that for each $w \in Q_I^{up}$, the disk $\beta(w,R)$ is contained in $\beta(c_I, 2 R) \supset Q_I^{up}$ and $A(\beta(c_I, 2 R)) \simeq A(Q_I^{up})$.  We estimate, using the mean value property for $f$, Proposition \ref{diskslattice}, Jensen's inequality, Fubini, and the finite overlap condition:
\begin{align*}
\int_{\R_{+}^2} |f(w)|^2 |b(w)|^2 \, dA_\alpha(w) & = \sum_{I \in \calD} \int_{Q_I^{up}} \left|\frac{1}{A(\beta(w, R))} \int_{\beta(w, R)} f(\zeta) \, dA (\zeta)\right |^2 |b(w)|^2 \, dA_\alpha(w)\\
& \lesssim  \sum_{I \in \calD} \int_{\beta(c_I, 2 R)} \frac{1}{A(\beta(c_I, 2R))} \int_{\beta(c_I, 2R)} |f(\zeta)|^2 |b(w)|^2 \, dA(\zeta) dA_\alpha(w) \\
& = \sum_{I \in \calD} \int_{\beta(c_I, 2 R)} \left(\frac{1}{A(\beta(c_I, 2R))} \int_{\beta(c_I, 2R)} |b(w)|^2 \, dA(w)\right)  |f(\zeta)|^2  dA_\alpha(\zeta)  \\
& \leq \|b\|_{\textnormal{BA}_{2R}} \cdot  \sum_{I \in \calD} \int_{\beta(c_I, 2 R)}  |f(\zeta)|^2  dA_\alpha(\zeta)\\
& \lesssim \|b\|_{\textnormal{BA}_{2R}} \cdot \|f\|^2_{A^2_\alpha(\R_+^2)}.
\end{align*}
The proof is complete.
\end{proof}

Finally, we define the $\calB_\infty(\R_+^2)$ class. 
\begin{defn}
We say a weight $\sigma$ belongs to $\calB_\infty(\R_+^2)$ if
$$ 
[\sigma]_{\calB_\infty}:=\sup_{I \in \calD} \frac{1}{\sigma(Q_I)} \int_{Q_I} M_{\calD}(\sigma \mathbf{1}_{Q_I}) \, dA_\alpha< \infty.$$ 
\end{defn}

\subsection{Boundedness Results} 

Given a measurable function $ \varphi$ on $L^2_\alpha(\R_{+}^2)$, define the \emph{Hankel operator}
$$ 
H_\varphi f:= (I-P_\alpha)M_\varphi P_\alpha f,  \quad f \in L^2_\alpha(\R_{+}^2).
$$
Here, $M_\varphi$ denotes the \emph{multiplication operator} on $L^2_\alpha(\R_{+}^2)$ with symbol $\varphi.$ Note that $H_\varphi$ is densely defined for bounded, compactly supported $f$ and $\varphi \in {\rm BMO}^2(\R_+^2)$. To prove Theorem \ref{mainthm001dropanalytic}, we will make use of a well-known link between Hankel operators and commutators.  

\begin{proof}[Proof of Theorem \ref{mainthm001dropanalytic}]
First, note that $b \in {\rm BMO}^2$ implies $\overline{b} \in {\rm BMO}^2$, and the Hankel operator $H_{\overline{b}}$ is bounded on $L^2_\alpha(\R_{+}^2, \sigma)$ if and only the adjoint operator $H_{\overline{b}}^*$ is bounded on $L^2_\alpha(\R_{+}^2, \sigma^{-1})$ . Here, the adjoint is computed with respect to the standard (unweighted) $L^2_\alpha$ pairing. Moreover, on $L^2_\alpha$ we have the operator identity
$$ 
[b, P_\alpha]= H_b -H_{\overline{b}}^*.
$$
Finally, $\sigma \in \mathcal{B}_2 \cap \mathcal{APR}$ if and only if $\sigma^{-1} \in \mathcal{B}_2 \cap \mathcal{APR}$,
so to prove $[b,P_\alpha]$ extends to a bounded operator on $L^2_\alpha(\R_{+}^2, \sigma)$ for any $b \in {\rm BMO }_2$, $\sigma \in \calB_2 \cap\mathcal{APR} $, it is enough to prove $H_b$ extends to a bounded operator on $L^2_\alpha(\R_{+}^2, \sigma)$  under identical assumptions. 

As in the unweighted proof in the case of the unit disk in \cite{Zhu07}, we split the symbol $b=b_1+b_2$ according to Proposition \ref{splitting}, where $b_1 \in \textnormal{BO}(\R_+^2)$ and $b_2 \in \textnormal{BA}(\R_+^2),$ and write $H_b= H_{b_1}+ H_{b_2}.$ We note that this splitting can be accomplished by setting $b_2=b-b_1$, where
$$
b_1(z):= \frac{1}{A(\beta(z,r))} \int_{\beta(z,r)} b(w) \, dA(w).
$$
We handle the boundedness of each term separately.

For $H_{b_1}$, we note that by the structure of $b_1$, it is clear that the crucial oscillation lemma 
(Lemma \ref{OscLemma}) will hold true, even though $b_1$ is not holomorphic. That is the only time we made use of the holomorphic condition on $b_1$ (specifically, the mean-value property) in the proof of Theorem \ref{mainthm001}. Therefore, one could actually follow the proof of Theorem \ref{mainthm001} in the special case $\mu=\lambda= \sigma$ to obtain the desired bound for $H_{b_1}.$  We present a totally alternative proof here to bound $H_{b_1}$, which we believe may be of independent interest. This proof involves generalizations of Bergman-type operators, which were studied using similar dyadic methods in the case of the unit ball in \cite{RTW2017}.

Let $f \in L^2_\alpha(\R_{+}^2, \sigma).$ Then $P_\alpha f \in L^2_\alpha(\R_{+}^2, \sigma),$ so in particular $g:=P_\alpha f$ is holomorphic on the upper-half plane. Using this fact, the hypothesis that $b_1 \in \textnormal{BO}(\R_+^2)$ as well as Proposition \ref{Lip}, and the remark thereafter:
\begin{align*}
|H_{b_1} f(z)| & = \left| \int_{\R_{+}^2} \frac{(b_1(z)-b_1(w)) g(w)}{(z-\overline{w})^{2+\alpha}} \, dA_\alpha(w) \right|\\
& \leq \|b_1\|_{\textnormal{BO}}   \cdot  \int_{\R_{+}^2} \frac{(d_\beta(z,w)+1) |g(w)|}{|z-\overline{w}|^{2+\alpha} }\, dA_\alpha(w)\\
& \leq  \|b_1\|_{\textnormal{BO}} \cdot  \left( C_\varepsilon \int_{\R_{+}^2} \frac{ (\im z)^{-\varepsilon} (\im w)^{-\varepsilon} |g(w)|}{|z-\overline{w}|^{2+\alpha-2 \varepsilon}} \, dA_\alpha(w) + \int_{\R_{+}^2} \frac{ |g(w)|}{|z-\overline{w}|^{2+\alpha}} \, dA_\alpha(w) \right)\\
& =: \|b_1\|_{\textnormal{BO}} \cdot \left( C_\varepsilon \mathcal{Q}_\alpha^\varepsilon g(z) + P_\alpha^{+}g(z) \right).
\end{align*}
 Taking $L^2_\alpha(\R_{+}^2,\sigma)$ norms, we obtain  
$$ 
\|H_{b} f\|_{L^2_\alpha(\R_{+}^2,\sigma)} \lesssim \|b_1\|_{\textnormal{BO}} \cdot  \left(\|Q^{\varepsilon}_{\alpha}g\|_{L^2_\alpha(\R_{+}^2,\sigma)}+ \|P_\alpha^{+} g\|_{L^2_\alpha(\R_{+}^2,\sigma)}\right).
$$
By \cite[Theorem 1]{RTW2017}, it is clear that $\|P_\alpha^+g\|_{L^2_\alpha(\R_{+}^2,\sigma)} \lesssim [\sigma]_{\calB_2}^2 \|f\|_{L^2_\alpha(\R_{+}^2,\sigma)}$  from our assumption that $\sigma \in \calB_2$. Therefore, it suffices to prove $$ \|Q^{\varepsilon}_{\alpha}g\|_{L^2_\alpha(\R_{+}^2,\sigma)} \leq C_{\sigma,\varepsilon} \|f\|_{L^2_\alpha(\R_{+}^2,\sigma)}$$ for an appropriately small choice of $\varepsilon>0.$

For notational convenience, define the function $\rho_\varepsilon(w)= (\im\,w)^{-\varepsilon}, w \in \R_{+}^2$. Following the same strategy of proof of Proposition \ref{BergSparse}, it is straightforward to show one can achieve the domination
\begin{equation} \label{ModSparse}
\mathcal{Q}_\alpha^\varepsilon g(z) \lesssim \sum_{j=1,2} \sum_{I \in \mathcal{D}^j} \left(\frac{1}{A_{\alpha-2 \varepsilon}(Q_I)} \int_{Q_I} g \, \rho_\varepsilon \, dA_\alpha \right) \rho_\varepsilon(z) \mathbf{1}_{Q_I}(z). 
\end{equation}
Suppose we show the $\calB_2$-like condition
\begin{equation} \label{FurtherWeightedB2}
C_{\sigma,\varepsilon}:= \sup_{I \in \calD} \langle \rho_\varepsilon \sigma \rangle_{Q_I}^{dA_{\alpha-\varepsilon}} \langle \rho_\varepsilon^{-1} \sigma^{-1} \rangle_{Q_I}^{dA_{\alpha-3\varepsilon}} <\infty.
\end{equation}
Then we can estimate, using \eqref{ModSparse} for any $h \in L^2_\alpha(\R_{+}^2,\sigma)$ of norm $1$: 

\begin{align*}
| \langle \mathcal{Q}_\alpha^\varepsilon g, \sigma h \rangle_{dA_\alpha}| & \lesssim \sum_{j=1,2} \sum_{I \in \calD^j} \langle g \rho_\varepsilon \rangle_{Q_I}^{dA_\alpha} \langle h\rho_\varepsilon \sigma \rangle_{Q_I}^{dA_\alpha} A_{\alpha+2 \varepsilon}(Q_I) \\
& = \sum_{j=1,2} \sum_{I \in \calD^j} \langle g \sigma \rho_\varepsilon^{-1}  \rangle_{Q_I}^{\rho_\varepsilon^{-1}\sigma^{-1} dA_{\alpha-3 \varepsilon}} \langle h \rho_{\varepsilon}^{-1}  \rangle_{Q_I}^{\rho_\varepsilon \sigma dA_{\alpha-\varepsilon}} A_{\alpha-2 \varepsilon}(Q_I)  \left(  \langle\rho_\varepsilon^{-1}\sigma^{-1} \rangle_{Q_I}^{ dA_{\alpha-3 \varepsilon}} \langle \rho_\varepsilon \sigma   \rangle_{Q_I}^{dA_{\alpha-\varepsilon}} \right)\\
& \leq C_{\sigma, \varepsilon} \sum_{j=1,2} \sum_{I \in \calD^j} \langle g \sigma \rho_\varepsilon^{-1}  \rangle_{Q_I}^{\sigma^{-1} dA_{\alpha-2 \varepsilon}} \langle h \rho_{\varepsilon}^{-1}  \rangle_{Q_I}^{\sigma dA_{\alpha-2\varepsilon}} A_{\alpha-2 \varepsilon}(Q_I)  ,
\end{align*}
and this expression can be controlled by $\|g\|_{L^2_\alpha(\R_{+}^2, \sigma)}$ via Cauchy--Schwarz and a standard maximal function argument which we omit. The proof is thus complete up to the verification of \eqref{FurtherWeightedB2}.

 \medskip 

Note that since $\sigma \in \calB_2 \cap \mathcal{APR}$, both $\sigma$ and $\sigma^{-1}$ satisfy a reverse H\"{o}lder inequality (see,  \cite{APR2019}). We estimate each factor of expression \eqref{FurtherWeightedB2} in turn, letting $\tau_\sigma>1$ being an exponent such that $\sigma, \sigma^{-1}$ satisfy a reverse H\"{o}lder inequality with respect to $\tau_\sigma$ and applying H\"{o}lder's inequality with exponents $\tau_\sigma$ and $\tau_\sigma'$ (assume $\varepsilon$ is chosen small enough so $\varepsilon \tau_{\sigma}'<1$):
\begin{align*}
\langle \rho_\varepsilon \sigma \rangle_{Q_I}^{dA_{\alpha-\varepsilon}} & \simeq \frac{1}{(\im\,c_I)^{2+\alpha- \varepsilon}} \int_{Q_I} \sigma(w)(\im\, w)^{-2 \varepsilon} \, dA_\alpha(w) \\
& \leq \left(\frac{1}{(\im\,c_I)^{2+\alpha}} \int_{Q_I} \sigma ^{\tau_\sigma} \, dA_\alpha \right)^{1/ \tau_\sigma} \times \left(\frac{1}{(\im\,c_I)^{2+\alpha- \varepsilon \tau_\sigma'}} \int_{Q_I} (\im\,w)^{- 2\varepsilon \tau_\sigma'} \, dA_\alpha(w)    \right)^{1/ \tau_\sigma'}\\
& \lesssim  \left(\frac{1}{(\im\,c_I)^{2+\alpha}} \int_{Q_I} \sigma ^{\tau_\sigma} \, dA_\alpha \right)^{1/ \tau_\sigma} \times \left( \frac{\left(\im c_I \right)^{2+\alpha-2\epsilon \tau_\sigma'}}{(\im c_I)^{2+\alpha-\epsilon \tau_\sigma'}} \right)^{1/\eta_\sigma'} \\ 
& \lesssim [\sigma]_{\rm{RH}_{\tau_\sigma}} \langle \sigma\rangle_{Q_I} (\im\, c_I)^{-\varepsilon}.
\end{align*}
On the other hand, 
\begin{align*}
\langle \rho_\varepsilon^{-1} \sigma^{-1} \rangle_{Q_I}^{dA_{\alpha-3\varepsilon}} & \simeq \frac{1}{(\im \, c_I)^{2+\alpha-3 \varepsilon}} \int_{Q_I}\sigma^{-1}(w) (\im \, w)^{-2 \varepsilon} \, dA_\alpha(w) \\
& \leq \left(\frac{1}{(\im \, c_I)^{2+\alpha}} \int_{Q_I} \sigma ^{-\tau_\sigma} \, dA_\alpha \right)^{1/ \tau_\sigma} \times \left(\frac{1}{(\im \, c_I)^{2+\alpha-3 \tau_\sigma'\varepsilon}} \int_{Q_I} (\im \, w)^{-2 \varepsilon \tau_\sigma'} \, dA_\alpha(w)    \right)^{1/ \tau_\sigma'}\\
& \lesssim [\sigma^{-1}]_{\rm{RH}_{\tau_\sigma}} \langle \sigma^{-1} \rangle_{Q_I} (\im\, c_I)^\varepsilon.
\end{align*}
These estimates show that \eqref{FurtherWeightedB2} is controlled by a constant times $[\sigma]_{\calB_2}$, which altogether proves the estimate 
 $$ 
 \|H_{b_1} f\|_{L^2_\alpha(\R_{+}^2,\sigma)} \lesssim [\sigma]_{\calB_2} \|g\|_{L^2_\alpha(\R_{+}^2,\sigma)} + [\sigma]_{\calB_2}^2 \|f\|_{L^2_\alpha(\R_{+}^2,\sigma)} \lesssim  [\sigma]_{\calB_2}^2 \|f\|_{L^2_\alpha(\R_{+}^2,\sigma)}.
 $$
It remains to show that $H_{b_2}$ is a bounded map on $L^2_\alpha(\R_{+}^2,\sigma)$. We have for $f \in L^2_\alpha(\R_{+}^2, \sigma)$ and $g:= P_\alpha f \in A^2_\alpha(\R_{+}^2, \sigma)$, using the fact that $P_\alpha$ is a bounded map on $L^2_\alpha(\R_{+}^2, \sigma)$ for $\sigma \in \calB_2$:
\begin{align*}
\|(I-P_\alpha) b_2 g\|_{L^2_\alpha(\R_{+}^2, \sigma)} & \lesssim (1+[\sigma]_{\calB_2}) \|b_2 g\|_{L^2_\alpha(\R_{+}^2, \sigma)}\\
& = (1+[\sigma]_{B_2}) \left(\int_{\R_{+}^2} |b_2|^2 |g|^2 \sigma \, dA_\alpha \right)^{1/2}.
\end{align*} 
Therefore, it remains to show that $|b_2|^2 \sigma dA$ is a Carleson measure on $A^2_\alpha(\R_{+}^2,\sigma)$. Note that since $b_2 \in \textnormal{BA}(\R_+^2)$, we already know that $|b_2|^2 dA_\alpha$ is a Carleson measure on the unweighted Bergman space $A^2_\alpha(\R_{+}^2)$ by Proposition \ref{Carleson}. We estimate as follows, using the $\mathcal{APR}$ condition on $\sigma$, Proposition \ref{diskslattice}, the subharmonicity of $|f|^2$, and the fact that $\im z$ is more or less constant on balls of fixed radius in the Bergman metric:
\begin{align*}
&\int_{\R_{+}^2} |b_2(w)|^2 |g(w)|^2 \sigma(w) \, dA_\alpha(w)\\
 & \simeq \sum_{I \in \calD} \sigma(c_I) \int_{Q_I^{up}} |b_2(w)|^2 |g(w)|^2  \, dA_\alpha(w)\\
& \lesssim \sum_{I \in \calD} \sigma(c_I) \left(\frac{1}{A(\beta(c_I,2R))}\int_{\beta(c_I,2R)} |g(\zeta)|^2 \, dA(\zeta)\right)\int_{Q_I^{up}} |b_2(w)|^2  \, dA_\alpha(w)\\
& \lesssim \left\|b_2\right\|^2_{\textnormal{BA}_{2R}} \cdot \sum_{I \in \calD} \sigma(c_I) \left(\frac{1}{A(\beta(c_I,2R))}\int_{\beta(c_I,2R)} |g(\zeta)|^2 \, dA(\zeta)\right) \times A_\alpha(\beta(c_I,2R)) \\
& \lesssim  \left\|b_2\right\|^2_{\textnormal{BA}_{2R}} \cdot  \sum_{I \in \calD} \sigma(c_I) \int_{\beta(c_I,2R)} |g(\zeta)|^2 \, dA_\alpha(\zeta)\\
& \simeq \left\|b_2\right\|^2_{\textnormal{BA}_{2R}} \cdot  \sum_{I \in \calD} \int_{\beta(c_I,2R)} |g(\zeta)|^2 \sigma(\zeta) \, dA_\alpha(\zeta) \\
& \lesssim \left\|b_2\right\|^2_{\textnormal{BA}_{2R}} \cdot  \|g\|_{L^2_\alpha(\R_{+}^2,\sigma)}^2\\
&\lesssim \left\|b_2\right\|^2_{\textnormal{BA}_{2R}} \cdot [\sigma]_{\calB_2}^4 \left\|f \right\|^2_{L_\alpha^2(\R_+^2, \sigma)}. 
\end{align*}
This estimate completes the proof. 
\end{proof}

For the operator $H_{b_1}$, with $b_1 \in \textnormal{BO}(\R_+^2)$, the extra assumption that $\sigma$ satisfies the $\mathcal{APR}$ condition can in fact be removed. This result is interesting in its own right, so we provide the proof below.

\begin{thm} \label{DropRH1}
Let $b\in \textnormal{BO}(\R_+^2)$ and $\sigma \in \calB_2.$ Then $H_b$ extends to a bounded operator from $L^2_\alpha(\R_{+}^2, \sigma)$ to $L^2_\alpha(\R_{+}^2, \sigma)$ and satisfies
$$
\|H_b\|_{L^2_\alpha(\R_{+}^2, \sigma) \rightarrow L^2_\alpha(\R_{+}^2, \sigma)} \lesssim_{\sigma} \|b\|_{\textnormal{BO}}. 
$$
\end{thm}

\begin{proof}
Examining the proof of Theorem \ref{mainthm001dropanalytic}, it suffices to show that the operator $\mathcal{Q}_\alpha^\varepsilon$ is bounded on $L^2_\alpha(\R_{+}^2, \sigma)$ for an appropriately small choice of $\varepsilon.$ As before, this is reduced to checking
$$
C_{\sigma,\varepsilon}:= \sup_{I \in \calD} \langle \rho_\varepsilon \sigma \rangle_{Q_I}^{dA_{\alpha-\varepsilon}} \langle \rho_\varepsilon^{-1} \sigma^{-1} \rangle_{Q_I}^{dA_{\alpha-3\varepsilon}} <\infty.$$

 Fix $I \in \calD^j$ arbitrary. We need to estimate these factors a bit more carefully this time. We concentrate on the factor $\langle \rho_\varepsilon \sigma \rangle_{Q_I}^{dA_{\alpha-\varepsilon}}.$  Decomposing $Q_I$ into a union of upper-halves $Q_J^{up}$, using properties of Carleson boxes, and applying H\"{o}lder's inequality we have, letting $\tau>1$ an exponent to be chosen later:
\begin{align*}
& \frac{1}{(\im\, c_I)^{2+\alpha-\varepsilon}} \int_{Q_I}\sigma(w) (\im\, w)^{-2 \varepsilon} \, dA_\alpha(w)\\
& = \frac{1}{(\im \, c_I)^{2+\alpha- \varepsilon}} \sum_{J \subseteq I}  \int_{Q_{J}^{up}} \sigma(w) (\im\,w)^{-2 \varepsilon} \, dA_\alpha(w)\\
& \lesssim \frac{1}{(\im \, c_I)^{2+\alpha- \varepsilon}} \sum_{J \subseteq I} (\im \, c_J)^{-2 \varepsilon} A_\alpha(Q_J^{up}) \left(\frac{1}{A_\alpha(Q_{J}^{up})} \int_{Q_{J}^{up}} \sigma(w) \, dA_\alpha(w) \right)\\
& \leq \frac{1}{(\im \, c_I)^{2+\alpha- \varepsilon}} \left(\sum_{J \subseteq I}(\im\, c_J)^{-2 \varepsilon \tau'} A_\alpha(Q_J^{up}) \right)^{1/\tau'} \left(\sum_{J \subseteq I} A_\alpha(Q_J^{up})\left(\frac{1}{A_\alpha(Q_J^{up})} \int_{Q_J^{up}} \sigma(w) \, dA_\alpha(w) \right)^\tau \right)^{1/\tau}\\
& \lesssim \frac{1}{(\im \, c_I)^{2+\alpha- \varepsilon}} (\im\, c_I)^{(2+\alpha-2 \varepsilon \tau')/\tau'}\left(\sum_{J \subseteq I} A_\alpha(Q_J^{up}) \left[\langle \sigma \rangle_{Q_{J}^{up}}\right] ^\tau \right)^{1/\tau}\\
& \lesssim (\im\, c_I)^{-\varepsilon} \left( \frac{1}{A_\alpha(Q_I)}\sum_{J \subseteq I} A_\alpha(Q_J^{up})\left[\langle \sigma \rangle_{Q_{J}^{up}}\right] ^\tau \right)^{1/\tau}.
\end{align*}
Define the dyadic regularization of $\sigma$ localized to $I$: 
$$ 
\sigma_I= \sum_{J \subseteq I} \langle \sigma \rangle_{Q_{J}^{up}} \mathbf{1}_{Q_{J}^{up}}.
$$

Then the last display is equal to 

$$(\im\, c_I)^{-\varepsilon} \left( \frac{1}{A_\alpha(Q_I)} \int_{Q_I} |\sigma_I|^\tau \, dA_\alpha \right)^{1/\tau}.$$

If we can show that $\sigma_I$ satisfies a reverse H\"{o}lder inequality \emph{on the specific cube $Q_I$} with some exponent $\tau_\sigma$ depending only on $\sigma$ and constant depending on $\sigma$, not $I$, then we will be done. For then
\begin{align*}
\left( \frac{1}{A_\alpha(Q_I)} \int_{Q_I} |\sigma_I|^\tau \, dA_\alpha \right)^{1/\tau} & \lesssim \frac{1}{A_\alpha(Q_I)} \int_{Q_I} |\sigma_I| \, dA_\alpha= \langle \sigma \rangle_{Q_I}.
\end{align*}
and we can proceed as before. The estimate for the other factor will follow in the same way.

It thus suffices to prove the claim for $\sigma_I.$ First, we observe that for any $J \in \mathcal{D}(I)$, there holds $\int_{Q_{J}} \sigma_I \, dA_\alpha= \int_{Q_{J}}\, \sigma\,  dA_\alpha.$ Moreover, using the fact that $\sigma \in \calB_2$, it is easy to verify $\int_{Q_{J}} \sigma_I^{-1} \, dA_\alpha \leq \int_{Q_{J}} \sigma^{-1} \, dA_\alpha.$ As a consequence, we essentially obtain that $\sigma_I$ is a $\calB_2$ weight for ``dyadic descendants'' of $Q_I,$ or precisely
$$ 
\sup_{J \in \mathcal{D}(I)} \langle \sigma_I \rangle_{Q_{J}} \langle \sigma_I^{-1} \rangle_{Q_{J}} \leq [\sigma]_{\calB_2}< \infty.
$$
We use a similar argument as \cite{StockdaleWagner2023} to establish the reverse H\"{o}lder inequality. Let $M_{Q_I}$ be the dyadic maximal function localized to $Q_I$:
$$ 
M_{Q_I}f(z):= \sup_{J \in \calD^j, J \subseteq I} \langle |f| \rangle_{Q_{J}} \mathbf{1}_{Q_{J}}(z).$$
By the comments above, it is clear that for almost every $z \in Q_I,$ we have $M_{Q_I} \sigma_I (z)= M_{Q_I} \sigma(z)$ . Then the fact that $\sigma \in \calB_2 \supset \calB_\infty$ together with the argument given in \cite[Lemma 2.8]{StockdaleWagner2023} imply that there exists $\tau_\sigma>0$ so that we have 
$$\langle (M_{Q_I}\sigma_I)^{1+\tau_\sigma} \rangle_{Q_I} \lesssim [\sigma]_{\calB_\infty} (\langle \sigma \rangle_{Q_I})^{1+\tau_\sigma} $$ 

Moreover, it is obvious from the locally constant property of $\sigma_I$ that there exists an independent constant $C>0$ so $\sigma_I (z) \leq C M_{Q_I} \sigma_I(z)$ for almost every $z \in Q_{I}.$ Then, we can simply estimate:
\begin{align*}
\langle \sigma_I^{1+\tau} \rangle_{Q_I}& \lesssim \langle (M_{Q_I} \sigma_I)^{1+\tau} \rangle_{Q_I} \lesssim [\sigma]_{B_\infty} (\langle \sigma_I \rangle_{Q_I})^{1+\tau},
\end{align*}
which establishes the reverse H\"{o}lder inequality and completes the proof of Theorem \ref{DropRH1}.
\end{proof}

In the special case that $b$ is holomorphic, the above theorem (for $H_{\overline{b}}$) reduces to a weighted generalization of the main result of Axler in \cite{Axler1986}, which we state as follows. We remark that in this case, $b$ belongs to an analog of the classical ``Bloch space'' $\mathcal{B}$, which was the case studied by Axler on the disk in \cite{Axler1986}. 

\begin{cor} 
Let $b \in \textnormal{BMO}^2 \left(\R_+^2 \right) \cap H(\R_+^2)$ and $\sigma \in \calB_2.$ Then $[b, P_\alpha]$ extends to a bounded operator on $L^2_\alpha(\R_{+}^2, \sigma).$
\end{cor}

\begin{proof}
We give two different ways to prove the above results. The first approach is to simply apply the Bloom type result Theorem \ref{mainthm001} with $\mu=\lambda$ there (with $\nu \equiv 1$ now). Observe that the $\textnormal{BMOA}_{\nu}$ there is weaker than the assumption $b \in \textnormal{BMO}^2 \left(\R_+^2 \right) \cap H(\R_+^2)$. 

Alternatively, since in this case $[b, P_\alpha]=-H_{\overline{b}}^*$ since $H_b \equiv 0$, one can apply Theorem \eqref{DropRH1} to $H_{\overline{b}}$ and the weight $\sigma^{-1}$ to deduce the desired result. Here, we have used the fact that if $b \in \textnormal{BMO}^2 \left(\R_+^2 \right) \cap H(\R_+^2)$, then $\overline{b} \in \textnormal{BO}(\R_+^2)$ (see, e.g.,  \cite{Si2022} or \cite{Zhu07}).
\end{proof}

We can obtain a partial converse to Theorem \ref{mainthm001dropanalytic}. To do so, we need to introduce one additional function space and lemma.

\begin{defn}
Given $r>0$, we say a function $b \in L^2(\R_{+}^2)$ belongs to $\textnormal{BDA}_r(\R_+^2)$ (bounded distance to analytic) if

$$\|b\|_{\textnormal{BDA}_r(\R_+^2)}:=\sup_{z \in \R_{+}^2} \inf_{h \in H(\R_+^2)} \left(\frac{1}{A(\beta(z,r))} \int_{\beta(z,r)} |b-h|^2 \, dA \right)^{1/2}< \infty. $$  
\end{defn}

The following lemma has basically the same proof as the one given in \cite[Theorem 6.1]{Li1994} for bounded strongly pseudoconvex domains.
\begin{lem} \label{BDAIntersect}
Suppose $b, \overline{b} \in \textnormal{BDA}_r(\R_+^2)$. Then $b \in \textnormal{BMO}^2(\R_+^2)$ and satisfies $$\|b\|_{\textnormal{BMO}^2(\R_+^2)} \lesssim \|b\|_{\textnormal{BDA}_r(\R_+^2)}+  \|\overline{b}\|_{\textnormal{BDA}_r(\R_+^2)}.$$
    
\end{lem}

\begin{thm} \label{PartialConverse}
Let $b \in L^2\left(\R_+^2 \right)$ and $\sigma \in \calB_2.$ Suppose  $[b,P_\alpha]$ extends to a bounded operator on $L^2_\alpha(\R_{+}^2, \sigma)$. Then there holds 
\begin{equation}
\sup_{I \in \calD} {1\over A(Q_I)}\int _{Q_I} |b(z)-b_{Q_I}| dA(z) <+\infty. \label{PartialLowerBound}
\end{equation}
Suppose additionally 
 $\sigma(z)=|h'(z)|^\eta$, where $h$ is a conformal map from $\R_{+}^2$ to some simply connected domain $\Omega \subsetneq \C$ In this case, we have the stronger condition $b \in \textnormal{BMO}^2(\R_{+}^2)$ and moreover $\|b\|_{\textnormal{BMO}^2(\R_+^2)} \approx \|[b,P_\alpha]\|_{L^2_\alpha(\sigma, \R_{+}^2) \rightarrow L^2_\alpha(\sigma, \R_+^2)}.$

\begin{proof}
The first assertion, valid for any  $\sigma \in \mathcal{B}_2$, follows from the proof of necessity in Theorem \ref{mainthm001}, taking $\mu=\lambda=\sigma$ (recall we do not need the analytic hypothesis on $b$ to make use of the complex median method). There is a slight difference here as the measure $dA_\alpha$ is replaced by $dA$, but one can check that the underlying BMO spaces defined by \eqref{PartialLowerBound} and its $\alpha$-weighted counterpart are the same with comparable norms. For example, this fact follows easily when one passes to the BMO norm computed over hyperbolic disks using Remark \ref{EquivalentBMO}.

Now suppose $\sigma(z)=|h'|^\eta$ for conformal $h$. By pulling back $h$ to the disk via the Cayley transform and applying a version of the Koebe distortion theorem, one can see that $h$ is of bounded hyperbolic oscillation. If $[b,P_\alpha]$ is bounded on $L^2_\alpha(\sigma, \R_{+}^2)$, then the Hankel operator $H_b$ is bounded on $L^2_\alpha(\sigma, \R_{+}^2)$ by the equation $[b,P_\alpha]P_\alpha=H_b$ and the weighted estimated for the Bergman projection for $\sigma \in \mathcal{B}_2$. A similar argument shows $H_{\overline{b}}^*$ (where $^*$ denotes the $L^2_\alpha$ adjoint) is bounded on $L^2_\alpha(\sigma, \R_{+}^2)$, so by duality $H_{\overline{b}}$ is bounded on $L^2_\alpha(\sigma^{-1}, \R_{+}^2)$, where $\sigma^{-1}$ has the same form as $\sigma$ (except a power of $-\eta$) and is still a $\mathcal{B}_2$ weight. We will show the weighted bound on $H_b$ implies $b \in \textnormal{BDA}_r(\R_+^2)$ with norm control \begin{equation}\|b\|_{\textnormal{BDA}_r(\R_+^2)} \lesssim \|H_b\|_{L^2_\alpha(\sigma, \R_{+}^2) \rightarrow L^2_\alpha(\sigma, \R_+^2)} \lesssim \|[b,P_\alpha]\|_{L^2_\alpha(\sigma, \R_{+}^2) \rightarrow L^2_\alpha(\sigma, \R_+^2)}, \label{BDANormControl}\end{equation} and then the same argument run for $H_{\overline{b}}$ and the dual weight will give $\overline{b} \in \textnormal{BDA}_r(\R_+^2)$. Then Lemma \ref{BDAIntersect} completes the proof.

We turn to showing $b \in \textnormal{BDA}_r(\R_+^2).$ Let $r>0$ be arbitrary. For any $z \in \R_+^2$, consider the test function $k_z^\sigma(w)= (h')^{-\eta/2}k_z(w)$, where $k_z(w)= c_\alpha \frac{(\text{Im}z)^{1+\alpha/2}}{(z-\bar{w})^{2+\alpha}}$ denotes the normalized reproducing kernel for $A^2_\alpha(\R_+^2)$. An easy computation shows that $k_z^\sigma$ belongs to $A^2_\alpha(\sigma, \R_+^2)$ with unit norm, and moreover, $b k_z^\sigma \in L^2_\alpha(\sigma, \R_+^2)$, since $\| \text{Im}(\cdot)^\alpha k_z(\cdot)\|_{L^\infty(\R_{+}^2)}<\infty$ and $b \in L^2(\R_+^2)$. Moreover, the function $h(\zeta):=\frac{P_\alpha(b k_z^\sigma)(\zeta)}{k_z^\sigma(\zeta)}$ is well-defined and holomorphic on $\R_+^2$. We compute

\begin{align*}
\|H_b\|_{L^2_\alpha(\sigma, \R_{+}^2) \rightarrow L^2_\alpha(\sigma, \R_+^2)}^2& \geq \|H_b(k_z^\sigma)\|_{L^2_\alpha(\sigma, \R_{+}^2)}^2 \\
& = \|b k_z^\sigma- P(b k_z^\sigma)\|_{L^2_\alpha(\sigma, \R_{+}^2)}^2\\
& = \int_{\R_{+}^2} |k_z^\sigma|^2 \left|b- \frac{P_\alpha  (b k_z^\sigma)}{k_z^\sigma} \right |^2 \sigma  \, dA_\alpha \\
& \gtrsim \frac{1}{A(\beta(z,r))} \int_{\beta(z,r)}  |b- h|^2 \, dA.
\end{align*}

Taking a sup over $z \in \R_+^2$ then completes the proof.
    
\end{proof}
    
\end{thm}

\begin{rem}
Such examples of $\calB_2$ weights arise naturally from conformal mapping arguments; see, for example, \cite{LS2004}.    
\end{rem}

\subsection{A Counterexample} \label{20250107subsec01}
Here, we provide a counterexample to show that the desired Carleson embedding may \emph{not} hold with a general $\calB_2$ weight. Indeed, Theorem \ref{mainthm001dropanalytic} fails for general $\calB_2$ weights. 

For simplicity, we let $\alpha=0$. For the weight, let $$\sigma(z)= \frac{1}{|\im \, z-\frac{1}{2}|^{1/2}}.$$
It is obvious that $\sigma^{-1}$ is locally integrable; for the local integrability of $\sigma$ it suffices to note that $\sigma$ is integrable on the Euclidean disk $D(a+\frac{i}{2},\frac{1}{4})$ where $a \in \R$:
\begin{align*}
\int_{D(a+\frac{i}{2},\frac{1}{4})}  \frac{1}{|\im\,z-\frac{1}{2}|^{1/2}} \, dA(z) &  =  \int_{D(0,\frac{1}{4})} \frac{1}{\left|\im\,\zeta\right|^{\frac{1}{2}}} \, dA(\zeta) < \infty.
\end{align*}
To verify that $\sigma$ is a $\calB_2$ weight, we first consider the case when $|I|\leq \frac{1}{4}.$ Then notice we have the pointwise estimate $\sqrt{2} \leq \sigma(z) \leq 2$ for all $z \in Q_I$. Therefore, if $|I| \leq \frac{1}{4},$ we easily see that $\langle \sigma \rangle_{Q_I } \leq 2$ and $\langle \sigma^{-1} \rangle_{Q_I } \leq \frac{1}{\sqrt{2}}$.

Next, we consider the case when $|I|>\frac{1}{4}.$ Then the interval $(0,4|I|)$ contains the value $\frac{1}{2}$, and we can estimate
\begin{align*}
\int_{Q_I} \sigma(z) \, dA(z) & \leq |I| \int_{0}^{4 |I|} \frac{1}{\sqrt{|y-\frac{1}{2}|}} \, dy  \\
& \leq 2 |I| \int_{0}^{4 |I|} \frac{1}{\sqrt{|y}|} \, dy \\
& \lesssim |I|^{3/2},
\end{align*}
where the implicit constant is absolute.
Entirely similarly, we can check $\int_{Q_I} \sigma(z)^{-1} \, dA(z) \lesssim |I|^{5/2},$ so that 
$$ 
\langle \sigma \rangle_{Q_I} \langle \sigma^{-1} \rangle_{Q_I} \lesssim 1,
$$ 
as required.

Note that $\sigma$ obviously fails to be an $\mathcal{APR}$ weight, and the above arguments actually established that $\sigma \, dA$ is a Carleson measure (as $\sup_{I \in \calD} \langle \sigma \rangle_{Q_I}<+\infty$)! However, it is clear that
$$
\int_{D(\frac{i}{2}, \frac{1}{4})} \sigma^{2} \, dA= + \infty.
$$
We now show that this example can lead to a case where $\sigma \in \calB_2$ and $b \in {\rm BMO}^2$, but $[b, P]$ is not bounded from $L^2(\R_{+}^2, \sigma)$ to $L^2(\R_{+}^2, \sigma).$ We let $\sigma(z)=\frac{1}{\left|\im\,z-\frac{1}{2}\right|^{\frac{1}{2}}}$ as above, and let $b(z)= \sigma^{1/2}(z) \chi_{D(\frac{i}{2},\frac{1}{4})}(z)$. By the computations above, it is clear that $|b|^2 dA$ is still a Carleson measure on $A^2\left(D\left(\frac{i}{2}, \frac{1}{4} \right) \right)$, so that $b \in \textnormal{BA}(\R_+^2) \subset \textnormal{BMO}^2(\R_+^2)$.  Note that the function $ f:=\mathbf{1}_{D(\frac{i}{2}, \frac{1}{4})} \in L^2(\R_{+}^2, \sigma)$, 
 so to prove the operator $[b,P]$ is unbounded it suffices to show $\|b P(f)-P(bf)\|_{L^2(\sigma)}^2= + \infty.$ To this end, we provide a pointwise estimate on $P(bf)$:
\begin{align*}
|P(bf)(z)| & \leq  \int_{D(\frac{i}{2},\frac{1}{4})} \frac{1}{|z-\overline{w}|^2} \frac{1}{|\im\, w-\frac{1}{2}|^{1/4}} \, dA(w) \\
& \leq 16 \int_{D(\frac{i}{2},\frac{1}{4})}  \frac{1}{|\im\, w-\frac{1}{2}|^{1/4}} \, dA(w)  \\
& \leq c,
\end{align*}
where $c$ is an independent constant.

 Next, note that if $z \in D(\frac{i}{2}, \frac{1}{4})$, we also have the lower bound, using the mean value property:
\begin{align*}
|P f(z)| & = \left| \int_{D(\frac{i}{2},\frac{1}{4})} \frac{1}{(z-\overline{w})^2}  \, dA(w) \right| = \frac{1}{|z+\frac{i}{2}|^2} \geq \frac{16}{25}.
\end{align*}
Then we have, altogether,
$$ \int_{D(\frac{i}{2},\frac{1}{4})} |b(z)|^2 |Pf(z)|^2 \sigma(z) \, dA(z)= + \infty, \quad \int_{D(\frac{i}{2},\frac{1}{4})}  |P(bf)(z)|^2 \sigma(z) \, dA(z)< + \infty,   $$
by which we deduce $\|b P(f)-P(bf)\|_{L^2(\R_+^2, \sigma)}^2= + \infty$ and hence obtain the desired contradiction.

\subsection{Compactness Results}

We now introduce vanishing mean oscillation analogs of the spaces introduced in the preceding section. 

\begin{defn}
Let $b \in \textnormal{BMO}_r^2(\R_+^2)$ for $r>0.$ We say $b\in \textnormal{VMO}_r^2(\R_+^2)$  if
$$ \lim_{\varepsilon \rightarrow 0^{+}} \sup_{\substack{z \in \R_{+}^2:\\ \im z \leq \varepsilon}} \left(\frac{1}{A(\beta(z,r))} \int_{\beta(z,r)} |b-  b_{\beta(z,r)}|^2 \, dA \right)^{1/2}=0 
$$
and 
$$ 
\lim_{R \rightarrow \infty} \sup_{\substack{z \in \R_{+}^2:\\ |z| \geq R}} \left(\frac{1}{A(\beta(z,r))} \int_{\beta(z,r)} |b- b_{\beta(z,r)}|^2 \, dA \right)^{1/2}=0.
$$
\end{defn}

\begin{defn}
Let $b \in \textnormal{BO}_r(\R_+^2)$ for $r>0.$ We say $b \in \rm{VO}_r(\R_+^2)$  if
$$ 
\lim_{\varepsilon \rightarrow 0^{+}} \sup_{\substack{z \in \R_{+}^2:\\ \im z \leq \varepsilon}} \sup_{w \in \beta(z,r)} |b(z)-b(w)|=0
$$ 
and 
$$ 
\lim_{R \rightarrow \infty} \sup_{\substack{z \in \R_{+}^2:\\ |z| \geq R}} \sup_{w \in \beta(z,r)} |b(z)-b(w)|=0. 
$$
\end{defn}

\begin{defn}
Let $b \in \textnormal{BA}_r(\R_+^2)$ for $r>0.$ We say $b \in \rm{VA}_r(\R_+^2)$  if
$$  \lim_{\varepsilon \rightarrow 0^{+}} \sup_{\substack{z \in \R_{+}^2:\\ \im z \leq \varepsilon}} \left(\frac{1}{A(\beta(z,r))} \int_{\beta(z,r)} |b|^2 \, dA \right)^{1/2}=0 $$ and
 $$ \lim_{R \rightarrow \infty} \sup_{\substack{z \in \R_{+}^2:\\ |z| \geq R}} \left(\frac{1}{A(\beta(z,r))} \int_{\beta(z,r)} |b|^2 \, dA \right)^{1/2}=0.
 $$
\end{defn}

Again, these function spaces are independent (as sets) of the parameter $r$, and we drop the explicit dependence on $r$ in what follows. Moreover, similar to the $\textnormal{BMO}^2(\R_+^2)$ case, we could replace the $\textnormal{VMO}^2(\R_+^2)$ conditions involving averages over Bergman balls by averages over Carleson boxes $Q_I$ (see again, \cite[Section 5]{HHLPW2024} for more details). We have the following splitting of $\textnormal{VMO}^2(\R_+^2)$ functions. 

\begin{prop}\label{VMO splitting}
Any function $b \in \textnormal{VMO}^2(\R_+^2)$ has a decomposition $b=b_1+b_2$ with $b_1 \in \rm{VO}(\R_+^2)$, $b_2 \in \rm{VA}(\R_+^2)$.

\end{prop}

This decomposition can be used to prove the expected compactness result for the commutator $[b,P_\alpha]$ on $L^2(\sigma)$ for $\sigma$ a $\calB_2$ weight with bounded hyperbolic oscillation. 

\begin{thm} \label{mainthm002dropanalytic}
Let $\sigma \in \calB_2 \cap \mathcal{APR}$. Then if $b \in {\rm VMO}^2(\R_{+}^2)$, the commutator $[b, P_\alpha]$ is compact from  $L_\alpha^2 \left(\R_{+}^2, \sigma \right)$ to $L_\alpha^2 \left(\R_{+}^2, \sigma \right) $. 
\end{thm}

\begin{proof}
We only sketch the argument, as many of the ideas have been seen before. As before, it is sufficient to check that the Hankel operator $H_b$ is compact on $L_\alpha^2 \left(\R_{+}^2, \sigma \right)$ under the same assumptions on $b, \sigma.$ For this, use the splitting given in Proposition \ref{VMO splitting} to write 
$$
H_b=H_{b_1}+H_{b_2}, 
$$
with $b_1 \in \textnormal{VO}(\R_+^2)$ and $b_2 \in \textnormal{VA}(\R_+^2)$. We can then show $H_{b_1}$ and $H_{b_2}$ are compact separately. For $H_{b_1},$ use the operator identity $H_{b_1}= [b_1, P_\alpha] P_\alpha$ to reduce the problem to checking the compactness of  $[b_1, P_\alpha]$. Then split $[b_1, P_\alpha]= \sum_{j=0}^3 [b_1, P_\alpha^j]$ for a small parameter $\eta$ as in the proof of Theorem \ref{mainthm002}. The operator $[b_1, P_\alpha^3]$ will be compact (in fact Hilbert--Schmidt), and this can be verified by a direct integral calculation as before. The operators $[b_1, P_\alpha^j]$, $j=0,1,2$ will all be bounded with small operator norm by using a similar argument as the one given in Theorem \ref{mainthm002} with sparse domination, together with the $b_1 \in VO$ condition. Note that $b_1$ satisfies the conclusion of Lemma \ref{OscLemma} in this case, so the sparse domination argument can be applied.

To handle $H_{b_2}$, it is enough to show that the multiplication operator $M_{b_2}$ is compact on the Bergman space $A^2_\alpha(\R_{+}^2, \sigma). $ This can be achieved via a vanishing Carleson measure type argument, and we omit the details. In particular, it is easy to show via Montel's Theorem that if $\{f_n\}$ converges weakly to $0$ in $A^2_\alpha(\R_{+}^2, \sigma)$, then $\{f_n\} $ converges to $0$ uniformly on compact subsets of $\R_{+}^2$. The smallness of averages of $|b_2|^2$ on top-halves $Q_I^{up}$ off such a compact subset then allows us to control the $L^2_\alpha(\sigma,  dA_\alpha)$ norm of $b_2 f_n$ by a small constant, imitating the earlier proof of boundedness when the symbol $b_2$ belonged to \rm{BA} in Theorem \ref{mainthm001dropanalytic}.  
\end{proof}

\medskip
\section{Unit ball analog} \label{s:5}
Finally, we extend our main results to the unit ball case, whose proof follows closely from the upper half plane case, and hence, we would like to leave the details to the interested reader. Moreover, one can also extend these results further to general domains by following the framework developed in \cite{GHK2022} and \cite{HWW2021}. We start with recalling some definitions first. 

Let $\B$ be the open unit ball in $\C^n$, $\partial \B$ be its boundary, and $H(\B)$ be the ball algebra of holomorphic functions. Recall that under the metric $d_{\partial \B}(z, w):=\left|1- \langle z, w \rangle \right|^{1/2}$, $\partial \B$ becomes a space of homogeneous type and hence there exists a finite collection of dyadic systems $\calD_i, i=1, \dots, N$ that play the role of $1/3$-trick. Again, write $\calD:=\bigcup_{i=1}^N \calD_i$. Now for each $S \in \calD_i$, one can associate it with a Carleson tent $Q_S$ and an upper Carleson tent $Q_S^{up}$, in the spirit of \eqref{20250121eq21} and \eqref{20250121eq22}, respectively. This is known as the \emph{Bergman tree} structure in $\B$ and we refer the reader \cite[Section 2]{RTW2017} for more details. With these, we define the weight classes needed in a natural way. 

\begin{defn} [Dyadic B\'ekolle--Bonami $\calB_2$ weights on $\B$]
Let $\mu$ be a weight on $\B$ and $\alpha>-1$, we say $\mu$ belongs to the \emph{dyadic B\'ekolle--Bonami $\calB_2$ weight class} on $\B$ if 
$$
    \left[\mu \right]_{\calB_2}:=\sup_{S \in \calD} \left( \frac{1}{V_\alpha(Q_S)} \int_{Q_S} \mu(z) dV_\alpha(z) \right) \left( \frac{1}{V_\alpha(Q_S)} \int_{Q_S} \mu^{-1}(z) dV_\alpha(z) \right)<+\infty,
$$
where $dV_\alpha(z):=c_\alpha (1-|z|^2)^\alpha dV(z)$ is the standard probability measure on $\B$ with $c_\alpha$ being the normalizing constant. 
\end{defn}

\begin{defn}
Suppose $\nu\in \calB_2$. For $b \in L_{ loc}^1(\B, dV_\alpha)$, we say $b \in {\rm BMO}_{\nu} \left(\B \right)$ if 
$$
\left\|b \right\|_{{\rm BMO}_{\nu}(\B)}:=\sup_{S \in \calD} {1\over \nu(Q_S)}\int _{Q_S} |b(z)-b_{Q_S}|dV_\alpha(z)<+\infty.
$$
Moreover, if $b \in H(\B)$, we say $b \in {\rm BMOA}_{\nu}\left(\B \right).$
\end{defn}

\begin{defn}
Suppose $\nu\in \calB_2$. For $b \in {\rm BMO}_{\nu}\left(\B \right)$, we say $b \in {\rm VMO}_{\nu}(\B)$ if 
$$
\lim_{\textrm{diam}(S) \rightarrow 0}\sup_{Q_S} {1\over \nu(Q_S)}\int_{Q_S} |b(z)-b_{Q_S}|dV_\alpha(z)= 0.
$$
If in addition to satisfying the above, $b$ is holomorphic, then we say $b \in {\rm VMOA}_{\nu}\left(\B \right)$. Here, $\textnormal{diam}(S):=\sup_{z, w \in S}|z-w|$ is the \emph{diameter} of the set $S$.
\end{defn}
Then for the case when $b$ is a holomorphic symbol, we have the following:

\begin{thm}
Let $\mu, \lambda \in \calB_2$ and set $\nu=\mu^{\frac{1}{2}}\lambda^{-\frac{1}{2}}$. Let further, $b \in L_{ loc}^1\left(\B, dV_\alpha \right) \cap H(\B) $. Then the commutators $[b, P_\alpha]$ and $[b, P^+_\alpha]$ are  
\begin{enumerate} 
\item [$\bullet$] bounded from  $L_\alpha^2 \left(\B, \mu \right)$ to $L_\alpha^2 \left(\B, \lambda \right) $ if and only if $b \in {\rm BMOA}_{\nu} \left(\B \right)$. Moreover, we have the following Bloom-type estimates: 
$$ 
\|[b, P_\alpha]\|_{L_\alpha^2 \left(\B, \mu \right) \mapsto L_\alpha^2 \left(\B, \lambda \right) }, \; \|[b, P^+_\alpha]\|_{L_\alpha^2 \left(\B, \mu \right) \mapsto L_\alpha^2 \left(\B, \lambda \right) } \simeq \|b\|_{ {\rm BMOA}_{\nu} \left(\B \right) }.
$$ 
\item [$\bullet$] compact from  $L_\alpha^2 \left(\B, \mu \right)$ to $L_\alpha^2 \left(\B, \lambda \right) $ if and only if $b \in {\rm VMOA}_{\nu}\left(\B \right)$. 
\end{enumerate}
\end{thm}

Next, we treat the case when $b$ is not holomorphic. 

\begin{defn}\label{def-BMO}
For $b \in L^2(\B)$, we say $b \in {\rm BMO}^2 \left(\B \right)$ if 
$$
\left\|b \right\|_{{\rm BMO}^{2}(\B)}:=\sup_{S \in \calD} \left({1\over V(Q_S)}\int _{Q_S} |b(z)-b_{Q_S}|^2 dV(z) \right)^{1/2}<+\infty.
$$
\end{defn}

\begin{defn} [Aleman--Pott--Regeura weights ($\mathcal{APR}$)]
Let $\mu$ be a weight on $\B$, we say $\mu$ is a \emph{Aleman--Pott--Regeura weight ($\mathcal{APR}$)} or is of \emph{bounded hyperbolic oscillation} if there exists some $C_\mu>0$, such that for any $S \in \mathcal{D}$ and $z, w \in \beta_{\B} (c_S, R)$, there holds
$$
C_{\mu}^{-1} \mu(w) \le \mu(z) \le C_\mu \mu(w). 
$$
Here, $\beta_\B(z, r) \supset Q_S^{up}$ is the Bergman metric on $\B$ centered at $z \in \B$ with radius $r>0$, and $R>0$ is the universal constant that was determined in Proposition \ref{diskslattice} (see, also \cite[Theorem 2.5]{HH2021}). 
\end{defn}

\begin{thm}
Let $\sigma \in \calB_2 \cap \mathcal{APR}$. Then if $b \in {\rm BMO}^2(\B)$, the commutator $[b, P_\alpha]$ is bounded from  $L_\alpha^2 \left(\B, \sigma \right)$ to $L_\alpha^2 \left(\B, \sigma \right) $.  Moreover, we have
$$ \|[b, P_\alpha]\|_{L_\alpha^2 \left(\B, \sigma \right) \to L_\alpha^2 \left(\B, \sigma \right) } \lesssim \|b\|_{ {\rm BMO}^{2} \left(\B \right) }.
$$
Finally, the $\mathcal{APR}$ condition is sharp in this case. 
\end{thm}


\bigskip 
\begin{bibdiv}
\begin{biblist}

\bib{APR2019}{article}{,
    AUTHOR = {Aleman, Alexandru},
    author={Pott, Sandra}, 
    author={Reguera, Mar\'ia Carmen},
     TITLE = {Characterizations of a limiting class {$B_\infty$} of
              {B}\'ekoll\'e-{B}onami weights},
   JOURNAL = {Rev. Mat. Iberoam.},
    VOLUME = {35},
      YEAR = {2019},
    NUMBER = {6},
     PAGES = {1677--1692},
}

\bib{Axler1986}{article}{,
    AUTHOR = {Axler, Sheldon},
     TITLE = {The {B}ergman space, the {B}loch space, and commutators of multiplication operators},
   JOURNAL = {Duke Math. J.},
    VOLUME = {53},
      YEAR = {1986},
    NUMBER = {2},
     PAGES = {315--332},
}

\bib{BBCZ1990}{article}{,
    AUTHOR = {B\'ekoll\'e, David},
    AUTHOR = {Berger, Charles A},
    AUTHOR = {Coburn, Lewis A.}, 
    AUTHOR ={Zhu, Kehe},
     TITLE = {B{MO} in the {B}ergman metric on bounded symmetric domains},
   JOURNAL = {J. Funct. Anal.},
  FJOURNAL = {Journal of Functional Analysis},
    VOLUME = {93},
      YEAR = {1990},
    NUMBER = {2},
     PAGES = {310--350},
}

\bib{Bloom1985}{article}{,
    AUTHOR = {Bloom, Steven},
     TITLE = {A commutator theorem and weighted {BMO}},
   JOURNAL = {Trans. Amer. Math. Soc.},
    VOLUME = {292},
      YEAR = {1985},
    NUMBER = {1},
     PAGES = {103--122},
}



\bib{GHK2022}{article}{,
author={Gan, Chun},
author={Hu, Bingyang}, 
author={Khan, Ilyas},
     TITLE = {Dyadic decomposition of convex domains of finite type and
              applications},
   JOURNAL = {Math. Z.},
    VOLUME = {301},
      YEAR = {2022},
    NUMBER = {2},
     PAGES = {1939--1962},
}

\bib{HLW2017}{article}{,
    AUTHOR = {Holmes, Irina},
    author={Lacey, Michael T.}, 
    author={Wick, Brett D.},
     TITLE = {Commutators in the two-weight setting},
   JOURNAL = {Math. Ann.},
  FJOURNAL = {Mathematische Annalen},
    VOLUME = {367},
      YEAR = {2017},
    NUMBER = {1-2},
     PAGES = {51--80},
}

\bib{HH2021}{article}{,
    author = {Hu, Bingyang},
    author={Huo, Zhenghui},
     TITLE = {Dyadic {C}arleson embedding and sparse domination of weighted
              composition operators on strictly pseudoconvex domains},
   JOURNAL = {Bull. Sci. Math.},
  FJOURNAL = {Bulletin des Sciences Math\'ematiques},
    VOLUME = {173},
      YEAR = {2021},
     PAGES = {Paper No. 103067, 32},
}

\bib{HHLPW2024}{article}{,
author={Hu, Bingyang}, 
author={Huo, Zhenghui}, 
author={Lanzani, Loredana}, 
author={Palencia, Kevin},
author={Wagner, Nathan A.}, 
     TITLE = {The commutator of the {B}ergman projection on strongly
              pseudoconvex domains with minimal smoothness},
   JOURNAL = {J. Funct. Anal.},
    VOLUME = {286},
      YEAR = {2024},
    NUMBER = {1},
     PAGES = {Paper No. 110177, 45},
}

\bib{HWW2021}{article}{,
    AUTHOR = {Huo, Zhenghui},
    author= {Wagner, Nathan A.},
    author={Wick, Brett D.},
     TITLE = {Bekoll\'e-{B}onami estimates on some pseudoconvex domains},
   JOURNAL = {Bull. Sci. Math.},
  FJOURNAL = {Bulletin des Sciences Math\'ematiques},
    VOLUME = {170},
      YEAR = {2021},
     PAGES = {Paper No. 102993, 36},
}

\bib{LaceyLi2022}{article}{, 
    AUTHOR = {Lacey, Michael},
    author={Li, Ji},
     TITLE = {Compactness of commutator of {R}iesz transforms in the two
              weight setting},
   JOURNAL = {J. Math. Anal. Appl.},
    VOLUME = {508},
      YEAR = {2022},
    NUMBER = {1},
     PAGES = {Paper No. 125869, 11},
}

\bib{Hytonen2021}{article}{,
    AUTHOR = {Hyt\"onen, Tuomas P.},
     TITLE = {The {$L^p$}-to-{$L^q$} boundedness of commutators with
              applications to the {J}acobian operator},
   JOURNAL = {J. Math. Pures Appl. (9)},
    VOLUME = {156},
      YEAR = {2021},
     PAGES = {351--391},
}

\bib{LS2004}{article}{,
    AUTHOR = {Lanzani, Loredana},  AUTHOR={Stein, Elias M.},
     TITLE = {Szeg\"o{} and {B}ergman projections on non-smooth planar
              domains},
   JOURNAL = {J. Geom. Anal.},
  FJOURNAL = {The Journal of Geometric Analysis},
    VOLUME = {14},
      YEAR = {2004},
    NUMBER = {1},
     PAGES = {63--86},
}

\bib{Lerner2017}{article}{,
author={Lerner, Andrei K.}, 
author={Ombrosi, Sheldy},
author={Rivera-R\'ios, Israel P.},
     TITLE = {On pointwise and weighted estimates for commutators of
              {C}alder\'on-{Z}ygmund operators},
   JOURNAL = {Adv. Math.},
    VOLUME = {319},
      YEAR = {2017},
     PAGES = {153--181},
}


\bib{Li1992}{article}{,
    AUTHOR = {Li, Huiping},
     TITLE = {B{MO}, {VMO} and {H}ankel operators on the {B}ergman space of
              strongly pseudoconvex domains},
   JOURNAL = {J. Funct. Anal.},
    VOLUME = {106},
      YEAR = {1992},
    NUMBER = {2},
     PAGES = {375--408},
}

\bib{Li1994}{article}{,
    AUTHOR = {Li, Huiping},
     TITLE = {Hankel operators on the {B}ergman spaces of strongly
              pseudoconvex domains},
   JOURNAL = {Integral Equations Operator Theory},
    VOLUME = {19},
      YEAR = {1994},
    NUMBER = {4},
     PAGES = {458--476},
}


\bib{PottReguera2013}{article}{,
    author={Pott, Sandra}, 
    author = {Reguera, Maria Carmen},
     TITLE = {Sharp {B}\'ekoll\'e{} estimates for the {B}ergman projection},
   JOURNAL = {J. Funct. Anal.},
    VOLUME = {265},
      YEAR = {2013},
    NUMBER = {12},
     PAGES = {3233--3244},
}



\bib{RTW2017}{article}{,
author={Rahm, Rob },
author={Tchoundja, Edgar },
    AUTHOR = {Wick, Brett D.},
     TITLE = {Weighted estimates for the {B}erezin transform and {B}ergman
              projection on the unit ball},
   JOURNAL = {Math. Z.},
    VOLUME = {286},
      YEAR = {2017},
    NUMBER = {3-4},
     PAGES = {1465--1478},
}

\bib{StockdaleWagner2023}{article}{, 
    AUTHOR = {Stockdale, Cody B.}, 
    author={Wagner, Nathan A.},
     TITLE = {Weighted theory of {T}oeplitz operators on the {B}ergman
              space},
   JOURNAL = {Math. Z.},
  FJOURNAL = {Mathematische Zeitschrift},
    VOLUME = {305},
      YEAR = {2023},
    NUMBER = {1},
     PAGES = {Paper No. 10, 29},
}

\bib{Si2022}{article}{,
    AUTHOR = {Si, Jiajia},
     TITLE = {B{MO} and {H}ankel operators on {B}ergman space of the
              {S}iegel upper half-space},
   JOURNAL = {Complex Anal. Oper. Theory},
    VOLUME = {16},
      YEAR = {2022},
    NUMBER = {2},
     PAGES = {Paper No. 24, 30},
}



\bib{Wei2024}{arXiv}{
  author={Wei, Zhenguo}, 
author={Zhang, Hao},
  title={Complex median method and Schatten class membership of commutators},
  date={2024},
  eprint={2411.05810},
  archiveprefix={arXiv},
  primaryclass={math.FA},
}


\bib{Zhu1992}{article}{,
    AUTHOR = {Zhu, Kehe},
     TITLE = {B{MO} and {H}ankel operators on {B}ergman spaces},
   JOURNAL = {Pacific J. Math.},
    VOLUME = {155},
      YEAR = {1992},
    NUMBER = {2},
     PAGES = {377--395},
}

\bib{Zhu07}{book}{,
    AUTHOR = {Zhu, Kehe},
     TITLE = {Operator theory in function spaces},
    SERIES = {Mathematical Surveys and Monographs},
    VOLUME = {138},
   EDITION = {Second},
 PUBLISHER = {American Mathematical Society, Providence, RI},
      YEAR = {2007},
     PAGES = {xvi+348},
}

\bib{Zimmer2023}{article}{,
    AUTHOR = {Zimmer, Andrew},
     TITLE = {Hankel operators on domains with bounded intrinsic geometry},
   JOURNAL = {J. Geom. Anal.},
    VOLUME = {33},
      YEAR = {2023},
    NUMBER = {6},
     PAGES = {Paper No. 176, 29},
}



\end{biblist}
\end{bibdiv}
\end{document}